\documentclass[review]{article}
\usepackage{lineno,hyperref}
\modulolinenumbers[5]
\usepackage{amsmath}

\usepackage{amssymb}
\usepackage{xcolor}
\usepackage{bbm}
\usepackage{amsthm}
\usepackage{graphicx}
\usepackage{subcaption}
\usepackage{multirow}%
\usepackage{amsmath,amsfonts}%
\usepackage{mathrsfs}%
\usepackage{xcolor}%
\usepackage{textcomp}%
\usepackage{manyfoot}%
\usepackage{booktabs}%
\usepackage{algorithm}%
\usepackage{algorithmicx}%
\usepackage{algpseudocode}%
\usepackage{listings}%
\usepackage{bbm}
\usepackage{subcaption}
\usepackage{geometry}

\newtheorem{theorem}{Theorem}
\newtheorem{lemma}{Lemma}
\newtheorem{cor}[theorem]{Corollary}%
\DeclareMathOperator*{\argmax}{arg\,max}

\begin{document}

\title{Parameter estimation for a hidden linear birth and death process with immigration  modeling disease propagation\thanks{This project was supported by the LabEx NUMEV (ANR 2011-LABX-076) within the I-SITE MUSE}}

\author{Ibrahim Bouzalmat\footnote{IMAG, Univ Montpellier, CNRS, Montpellier, France} \and Benoîte de Saporta\footnote{IMAG, Univ Montpellier, CNRS, Montpellier, France} \and Solym M. Manou-Abi\footnote{IMAG, Univ Montpellier, CNRS, Univ Mayotte, Montpellier, France}}

\date{}

\maketitle

\begin{abstract}
In this paper, we use a linear birth and death process with immigration to model infectious disease propagation when contamination stems from both person-to-person contact and contact with the environment. 
Our aim is to estimate the parameters of the process.
The main originality and difficulty comes from the observation scheme.
Counts of infected population are hidden.
The only data available are periodic cumulated new retired counts.
Although very common in epidemiology, this observation scheme is mathematically challenging even for such a standard stochastic process.
We first derive an analytic expression of the unknown parameters as functions of well-chosen discrete time transition probabilities.
Second, we extend and adapt the standard  Baum-Welch algorithm in order to estimate the said discrete time transition probabilities in our hidden data framework.
The performance of our estimators is illustrated both on synthetic data and real data of typhoid fever in Mayotte.
\end{abstract}

\section{Introduction}
%
Water  borne infectious  diseases \cite{ijerph7103657,ASHBOLT2004229} can be life-threatening and may carry infection to a large population in a very small duration. Typhoid fever, cholera, dysentery, poliomyelitis are, among others, the most frequent water borne diseases. Such diseases are a major health issue in most parts of the world and especially in developing countries. In the French administrated island of Mayotte, typhoid fever is a notifiable endemic disease with an average of 
$60$ cases per year between 2018 and 2022,  
under active surveillance of the Department of Safety and Health Emergencies (D\'eSUS) of the Regional Health Agency (ARS) of Mayotte. The disease is predominantly transmitted through contact with or consumption of microbially polluted water or food, see \cite{Salmonella,Salmonellabis,edelman1986summary}. 

Compartment models form a wide class of deterministic and random processes especially suitable to model disease propagation. A variety of such models are available in the literature for the transmission of typhoid fever with or without vaccination, see e.g. \cite{lauria2009optimization,mushayabasa2011impact,pitzer2014predicting,cvjetanovic1971epidemiological,watson2015review}.
In this paper, we focus on a simplified stochastic one-compartment model, where only counts of the infected population are taken into account. 
This stochastic model is a linear \emph{birth-and-death} process with \emph{immigration} (LBDI). Here \emph{immigration} corresponds to exogeneous contamination, or contamination from the environment; \emph{birth} corresponds to new person-to-person contamination; and \emph{death} corresponds to 
patient isolation (usually at hospital). 
Such processes belong to the family of Markov jump processes and are used in a wide variety of quantitative disciplines, such as evolutionary biology, epidemiology and queuing theory, see for instance \cite{codecco2008stochastic,novozhilov2006biological, nee2006birth,thorne1991evolutionary,crawford2014estimation}, and references therein.
This LBDI model can be described with only three scalar parameters: the \emph{immigration} (exogenous or environment contamination), \emph{birth} (person to person contamination) and \emph{death} 
(isolation) 
rates. Its asymptotic behavior is well-known as well as maximum likelihood estimators based on  both discrete and continuous observations if available observations include the jump times and post jump infected counts, see for instance  \cite{moran1953estimation,reynolds1973estimating,keiding1975maximum,perkins2009maximum}. 

However, in epidemiological practice continuous time or periodic counts of infected individuals are rarely available. For notifiable diseases, the most commonly collected data are cumulated new 
declaration 
counts over given periods, typically daily or weekly intervals. 
In our motivating example of typhoid fever, declaration is usually simultaneous to or quickly followed by hospitalisation and thus isolation of the patient.
Parameter inference for LBDI process in this observation framework is very challenging for two main reasons.

First, when observations are discrete, one needs to retrieve the transitions probabilities of infected counts over a given period. The immigration component invalidates the branching property, thus obtaining an analytical form for the transitions probabilities is not straightforward, and obtaining a numerical form may be liable to numerical instabilities \cite{immel1951problems}.
Several works have proposed approaches to estimate the parameters of discretely observed general birth and death processes. In \cite{crawford2014estimation,crawford2012transition}, the authors propose to use the theory of continued fractions to find the explicit expression of the Laplace transform of transition probabilities in order to consider the inference problem with discrete observations as a missing data problem by using the expectation-maximization (EM) algorithm to find maximum likelihood estimates. In \cite{xu2015likelihood}, the authors use the same EM algorithm method to propose a multi-type branching process approximation for birth-death shift processes. In \cite{davison2021parameter}, the authors propose to use two approaches to estimate the birth, death, and growth rates from a discretely observed linear birth-and-death process through an embedded Galton-Watson process and by maximizing a saddlepoint approximation to the likelihood. 

Second, in all the previously cited papers, observations are discrete but correspond to infected population counts. Retrieving infected counts from cumulated new isolated counts corresponds to a non standard hidden Markov model (HMM). Indeed, due to the underlying continuous-time dynamics, the cumulated new isolated over a period of time depend on both the infected counts at the beginning and at the end of the period. Therefore the emission distribution of the HMM 
depends on both the current and past states of the hidden chain. Hence the standard  Baum-Welch algorithm \cite{baum1966statistical} cannot be directly applied and has to be adapted to this specific case.

This paper addresses both challenges. First, we establish an analytical formula for the parameters of a LBDI process in terms of some discrete time transition probabilities. Then we adapt the Baum-Welch algorithm to our special form of emission distribution to estimate the transition probabilities of the hidden chain from the cumulated new isolated data. We finally obtain estimates of the original parameters by plug-in. Our results are illustrated both on synthetic and real data. The real data set concerns daily cumulated new declared confirmed cases of typhoid fever in Mayotte from 2018 to 2022. It was provided by the ARS. All codes were implemented in R and are available at \url{https://plmlab.math.cnrs.fr/bouzalma/hlbdi.git}.

The paper is organized as follows. We present the stochastic LBDI model and our specific observation framework in Section \ref{sec:model}. 
In Section \ref{sec:estim}, we present the parameter estimation procedure and establish the main properties of the estimators and provide the proofs in Section \ref{sec:proof}.
Section \ref{sec:num} is dedicated to the investigation of the numerical properties of the estimators, both on synthetic and real data.
Finally, Section \ref{sec:ccl} gathers some concluding remarks and future work directions.
%
\section{Model and observation framework}
\label{sec:model}
%
In this section, we first describe the stochastic model for counts of infected individuals. Then we define our observation framework and state our parameter estimation problem. Last, we explain our strategy to obtain estimators.
%
\subsection{Linear birth and death process with immigration}
\label{sec:LBDI}
%
For a water borne disease like typhoid fever, new infections are caused by direct or indirect contact with an infected individual or with the environment (contaminated food or water).
We make the following simplified assumptions.
Each individual contaminates a new one with a constant infection (birth) rate $\lambda$, independently from other individuals.
Infected individuals are isolated independently from one another at a constant isolation (death) rate $\mu$.
Exogenous new cases arrive at a constant (immigration) rate $\nu$ also independently from other contaminations and isolations.

Let $X_{t}$ be the number of infected individuals at time $t \geq 0$. Then the counting process $X= (X_{t})_{t\geq 0}$ is a linear birth and death process with immigration. It belongs to the class of Markov jump processes and its behavior is well understood, see  for instance \cite{feller1971introduction}. We recall the main properties here.
The distribution of $X$ is characterized by its infinitesimal generator $Q$ as follows
\begin{align*}
\mathbb{P}(X_{t+h}&=j |X_{t}=i) = Q(i,j)h + o(h)\quad\text{if } j\neq i,\\
\mathbb{P}(X_{t+h}&=i |X_{t}=i) = 1 + Q(i,i)h + o(h),
\end{align*}  
uniformly in $t$. 
The matrix $Q$ has non-zero entries given, for $i\geq 0$, by 
\begin{align*}
Q(i,i+1)&= \lambda i + \nu\\ 
Q(i,i-1) &= \mu i  \\ 
Q(i,i)&= - (\lambda + \mu) i - \nu.
\end{align*}  
The transition semi-group defined by $\big(P(t)=(p_{i,j}(t), (i,j)\in\mathbb{N}^2)\big)_{t\geq 0}$, where 
\begin{align*}
p_{i,j}(t)=\mathbb{P}(X_{t}=j |X_{0}=i), \; i, j \in \mathbb{N}, t\geq 0,
\end{align*}  
satisfies the (forward) Kolmogorov equation 
\begin{align}\label{eq1}
\frac{dP(t)}{dt}=P(t)Q,
\end{align}  
for $t>0$ with initial condition $P(0)$ equal to the Identity matrix on $\mathbb{N}$. It can be rewritten, for $i,j\in\mathbb{N}$ with $j\geq 1$, as
\begin{align*}
\frac{\mathrm{d} p_{i, 0}(t)}{\mathrm{d} t} &= \mu p_{i, 1}(t)- \nu  p_{i, 0}(t),\\
\frac{\mathrm{d} p_{i, j}(t)}{\mathrm{d} t} &=\left(\lambda (j-1) + \nu \right) p_{i, j-1}(t)+ \mu (j+1) p_{i, j+1}(t)-\left(( \lambda + \mu )j +\nu\right) p_{i, j}(t),
\end{align*}  
with initial condition $p_{i,i}(0) = 1$ and $ p_{i,j}=0$ for $j \neq i$. 
%
Finally, recall that a LBDI process (with $\nu>0$) has three possible regimes of long-time behavior: exponential growth (transience) if $\lambda>\mu$, all states are visited infinitely often with an infinite average return time (null recurrence) if $\lambda=\mu$, and all states are visited infinitely often with a finite average return time (positive recurrence) if $\lambda<\mu$. In the latter case, there exists a unique invariant distribution $\pi=(\pi_i, i\in\mathbb{N})$ given by 
\begin{align}\label{eq6}
\pi_{i}=
{r+i-1\choose i}
\left(\frac{\lambda}{\mu}\right)^{i}\left(1-\frac{\lambda}{\mu}\right)^r, 
\end{align}  
with $r=\nu/\lambda$.
In all the sequel, we assume that the LBDI process is in positive recurrent regime, i.e. $\nu>0$ and $\lambda<\mu$. Our aim is to estimate the triple $(\lambda, \mu,\nu)$.
\subsection{Cumulated new isolated counts}
\label{sec:cumul}
%
The main difficulty of this paper comes from the observation framework. We consider that $X_t$ is not observed. Only cumulated new isolated counts over given time lapses are available. Although this formulation is very common in public health data, it is mathematically original and challenging.

Let $Y_{(a,b]}$ denote the cumulated new isolated corresponding to process $X$ on the time interval $(a,b]$. Thus, $Y_{(a,b]}$ corresponds to the number of jumps with amplitude $-1$ of $X$ on the time interval $(a,b]$: $Y_{(a,b]}$ is the cardinal of the set $\{t\in(a,b]; X_t-X_{t^-}=-1\}$, where $X_{t^-}=\lim_{s<t, s\to t}X_s$.

The joint process $(X_t, Y_{(0,t]})$ is still a Markov jump process and its distribution is characterized by the transition probabilities
\begin{align*}
p_{i,(j, y)}(t) := \mathbb{P}\left(X_{t}=j,Y_{(0, t]}= y\mid X_{0}=i\right),
\end{align*}  
for $i,j,y\in\mathbb{N}$. They also satisfy a (backward) Kolmogorov equation, see Section \ref{app:KolmoY} for the proof.
%
\begin{lemma} \label{lem:KolmoY}
For $t>0$ and $i,j,y\in\mathbb{N}$, one has 
\begin{align}\label{eq:KolmoY}
\frac{\mathrm{d} p_{i,(j, y)}(t)}{\mathrm{d}t} =-((\lambda +\mu )i+\nu) p_{i,(j, y)}(t)+(\lambda i+\nu) p_{i+1,(j, y)}(t)+(\mu i) p_{i-1,(j, y-1)}(t),
\end{align}  
for $j \leq i+y$, with $p_{i,(j, y)}(t) = 0 $ whenever $j > i+y$ and with initial conditions  $p_{i,(j, y)}(0) = \delta_{i=j} \delta_{y=0}$.
\end{lemma}

Let $\Delta t$ be a fixed time lapse, typically one day or one week. Our aim is to estimate the triple $(\lambda, \mu,\nu)$ from observations given by the sequence $(Y_{(n\Delta t, (n+1)\Delta t]})_{n\in\mathbb{N}}$.
%
\subsection{Estimation strategy}
\label{sec:approach}
%
Our estimation problem falls into the class of hidden information problems, as one given sequence of observations $(Y_{(n\Delta t, (n+1)\Delta t]})_{n\in\mathbb{N}}$ may correspond to several different possible trajectories of $X$. However, our aim is not to reconstruct the most likely trajectory, but instead to estimate the triple $(\lambda, \mu,\nu)$.

Unfortunately, were not able to obtain an analytical form of the transitions probabilities $p_{i,(j, y)}(t)$ from Equation (\ref{eq:KolmoY}). Therefore we cannot access the likelihood of the observations, and there is no hope for a direct maximization with respect to the unknown parameters. Instead, we proceed in two steps.

The first step is to study the discrete time Markov chain $(X_{n\Delta t})_{n\in\mathbb{N}}$. We obtain an analytical closed form for its transition probability matrix in terms of the unknown parameters $(\lambda, \mu,\nu)$. This formula can be inverted to obtain an expression of the parameters $(\lambda, \mu,\nu)$ in terms of some well chosen transition probabilities. We then obtain estimators by plug-in from any estimator of the said transition probabilities, and we establish that consistence and asymptotic normality of the transition probability estimators transfer to the estimators of $(\lambda, \mu,\nu)$. 

The second step takes into account the missing information framework by interpreting the discrete time Markov chain $(X_{n\Delta t},Y_{(n\Delta t, (n+1)\Delta t]} )_{n\in\mathbb{N}}$ as a hidden Markov model, where $(X_{n\Delta t})$ is the hidden chain and $(Y_{(n\Delta t, (n+1)\Delta t]})$ the observations. However, in this framework the emission process is not standard as the cumulated number of new isolated $Y_{(n\Delta t, (n+1)\Delta t]}$ depends on the whole continuous time trajectory of $X$ on the interval $(n\Delta t, (n+1)\Delta t]$ and not only on its final value $X_{(n+1)\Delta t}$. Using the Markov property, one can obtain the emission probabilities of $Y_{(n\Delta t, (n+1)\Delta t]}$ conditionally to $X_{n\Delta t}$ and $X_{(n+1)\Delta t}$ only. Our first idea was to run the standard HMM procedure from the R package HMM on the two-dimensional chain $Z_n=(X_{n-1},X_n)$ to estimate the transition probabilities of the hidden chain $Z$. Unfortunately, the method was numerically unstable and marred by errors due to the many zero entries in the transition matrix of $Z$. Instead, we completely rewrote the forward and backward probabilities in our special context of sparse structured transition matrix and adapted the Baum Welsh algorithm accordingly. This second step enables us to obtain estimations of the hidden transition probabilities of the Markov chain $(X_{n\Delta t})_{n\in\mathbb{N}}$ based on the observations $(Y_{(n\Delta t, (n+1)\Delta t]})$ only. We then use the explicit formulas from step one to obtain estimators of the triple $(\lambda, \mu,\nu)$.
%
\section{Parameters estimation}
\label{sec:estim}
%
In this section, we detail the construction of our estimators in a discrete time framework. 
First, we establish an analytical formula for the parameters of a LBDI process in terms of some discrete time transition probabilities. Then we adapt the Baum-Welch algorithm to our special form of emission distribution to estimate the transition probabilities of the hidden chain from the cumulated new isolated data. We finally obtain estimates of the original parameters by plug-in.
%
\subsection{Analytical expression of the transition probabilities}
%
The first step of our estimation procedure is to establish an analytical formula for parameters $(\lambda, \mu,\nu)$ as functions of the discrete time transition probabilities of the Markov chain $(X_{n\Delta t})_{n\in\mathbb{N}}$.
We first establish an analytical solution to the Kolmogorov equation (\ref{eq1}). 
%
\begin{theorem}\label{main0}
For $t \geq 0$, the transition probabilities of the continuous-time Markov chain $X_{t}$ are given, for $i,j\in\mathbb{N}$, by
\begin{align}\label{eq2} 
\lefteqn{p_{i,j} (t) }\\
&= q(t)^r\!\!\!\!\!\! \sum_{l=0}^{min(i,j)} \!\!\!\!
\binom{i}{l}\!
\binom{r+i+j-l-1}{j-l} \!\!
\left(\frac{\mu}{\lambda}\right)^{i-l} \!\!\!\!\!\!
\left(1-q(t) \right)^{i+j-2l}\!\!
\left(\!{1-(\frac{\mu}{\lambda}+1)(1-q(t))}\!\right)^l\!\!\!, \nonumber 
\end{align}
where $r= \frac{\nu}{\lambda}$ and $q(t)=\frac{\mu -\lambda}{ \mu - \lambda e^{(\lambda - \mu) t }}$.
\end{theorem}

The proof is given in Section \ref{app:th1} and relies on turning the Kolmogorov equation into a standard first-order linear partial differential equation for the probability generating function of $X_t$. To the authors' best knowledge, the result of Theorem \ref{main0} is new. 
Note that a similar methodology was used and discussed in \cite{kobayashi2021stochastic1} yielding also closed form expressions for transition probabilities and originally comes from \cite{ross2014introduction}.

It is well known that the probability transition matrix of chain $(X_{n\Delta t})_{n\in\mathbb{N}}$ is $P(\Delta t)$, hence we have an analytic form for the discrete time transition probabilities. We now want to inverse Equation (\ref{eq2}) in order to obtain an expression of the parameters of interest $(\lambda, \mu,\nu)$ as a function of the transition probabilities. As there are only $3$ unknown parameters, it should be sufficient to select $3$ specific transition probabilities. In accordance with the real data set described in Section \ref{sec:dataARS}, we choose to use $p_{0,0}(\Delta t)$, $p_{0,1}(\Delta t)$ and $p_{1,0}(\Delta t)$ as they correspond to the most frequently observed transitions (of new isolated counts) and have the simplest analytical expressions. To simplify notations, in the sequel we will denote $p_{0,0}$, $p_{0,1}$ and $p_{1,0}$ instead of $p_{0,0}(\Delta t)$, $p_{0,1}(\Delta t)$ and $p_{1,0}(\Delta t)$, and mote generally $p_{i,j}$ instead of $p_{i,j}(\Delta t)$.
%
\begin{theorem}
\label{main1}
The parameters $\lambda$, $\mu$ and $\nu$ are given by the following relations 
\begin{equation}\label{eq:defg}
(\lambda,\mu,\nu)=g(p_{0,0}, p_{0,1},p_{1,0})
\end{equation}
 with $g=(g_1,g_2,g_3)$ and 
\begin{align*}
{\lambda} &=g_1(p_{0,0}, p_{0,1},p_{1,0})=\frac{\ln\left(\frac{u}{q}\right) ({q} - 1)}{\Delta t \left( {q} - {u} \right)},\\
{\mu} &=g_2(p_{0,0}, p_{0,1},p_{1,0})= \frac{\ln\left(\frac{{u}}{{q}}\right)( {u}- 1)}{\Delta t \left(  q- {u} \right)},\\
{\nu} &=g_3(p_{0,0}, p_{0,1},p_{1,0})= \frac{\ln(p_{0,0})}{\ln(q)}\frac{\ln\left(\frac{{u}}{{q}}\right)({q} - 1)}{\Delta t \left( {q} - {u} \right)},
\end{align*}
where 
\begin{align*}
q & =  \frac{p_{0,1}}{p_{0,0} \ln(p_{0,0})} W\left( \frac{p_{0,0}^{\left(\frac{p_{0,0}}{p_{0,1}} +1 \right)} \ln(p_{0,0})}{p_{0,1}} \right),\\
u & = 1-\frac{p_{1,0}}{p_{0,0}},
\end{align*}
and $W$ is the Lambert  function, see e.g. \cite{corless1993dj,corless1993lambert}.
\end{theorem}
%
The proof is given in Section \ref{app:th2} and relies on inversion of the formulas in Theorem \ref{main0}.
In particular, if one plugs estimators of $p_{0,0}$, $p_{0,1}$ and $p_{1,0}$ in the formulas of Theorem \ref{main1}, one obtains estimators of the unknown parameters $(\lambda, \mu,\nu)$.
%
\begin{cor}
\label{main2}
Let $(\hat p_{0,0}^n,\hat p_{0,1}^n,\hat p_{1,0}^n)$ be a consistent, asymptotically unbiased  and normal estimator of $(p_{0,0},p_{0,1},p_{1,0})$ in that 
\begin{align*}
\sqrt{n}\left(\begin{array}{c}
\hat p_{0,0}^n-p_{0,0}\\
\hat p_{0,1}^n-p_{0,1}\\
\hat p_{1,0}^n-p_{1,0}
\end{array}\right)
\xrightarrow[n\to\infty]{\mathcal{L}}\mathcal{N}_3(0,\Sigma'),
\end{align*}
for some covariance matrix $\Sigma'$. 
Then the estimators $(\hat \lambda^n,\hat \mu^n,\hat \nu^n)=g(\hat p_{0,0}^n,\hat p_{0,1}^n,\hat p_{1,0}^n)$ obtained by plug-in from Eq. (\ref{eq:defg}) are also consistent and asymptotically normal with asymptotic covariance matrix given by $\Sigma = Dg \Sigma' Dg^T$, where $Dg$ is the Jacobian matrix of mapping $g$ and $Dg^T$ is its transpose.
\end{cor}
%
The proof is straightforward as $g$ is continuously differentiable away from $0$. The computation of the Jacobian matrix of $Dg$ is detailed in Section \ref{app:D}.

For instance, let $\hat{p}^n_{i,j}$ be the maximum likelihood estimator of ${p}_{i,j}$ for observations $(X_{k\Delta t}, 0\leq k\leq n)$ and for $(i,j)\in\{(0,0), (0,1), (1,0)\}$. These estimators correspond to the ratio of the number of observed transitions from $i$ to $j$ over the total number of observed transitions from $i$ in the sequence $(X_{k\Delta t}, 0\leq k\leq n)$, see Section \ref{sec:num-simu-X}. Under the positive recurrence assumption ($\lambda<\mu$ and $\nu>0$), they are consistent and asymptotically unbiased and normal with limiting covariance matrix given by
\begin{align}\label{eq=sigma'}
\Sigma' =  \begin{pmatrix}
 \dfrac{p_{0,0} \left(1-p_{0,0}\right) }{\pi_0} & -\dfrac{p_{0,0}p_{0,1}}{\pi_0} & 0 \\
  -\dfrac{p_{0,0}p_{0,1}}{\pi_0} &\dfrac{p_{0,1} \left(1-p_{0,1}\right) }{\pi_0}  & 0 \\
   0 & 0 & \dfrac{p_{1,0} \left(1-p_{1,0}\right) }{\pi_1} \\
 \end{pmatrix},
\end{align}
where $\pi$ is the invariant distribution given in Equation (\ref{eq6}), 
see e.g.  \cite[Theorem 2.1]{teodorescu2009maximum}. Hence they satisfy the assumptions of Corollary \ref{main2}. However, they cannot be used in our context as the $(X_{k\Delta t}, 0\leq k\leq n)$ are hidden. Therefore we turn to another estimation procedure for $p_{0,0}$, $p_{0,1}$ and $p_{1,0}$ and use the framework of hidden Markov models.
%
\subsection{Hidden information framework}
\label{sec:HMM}
%
We now turn to the estimation of the transition probabilities $p_{0,0}$, $p_{0,1}$ and $p_{1,0}$ of the hidden chain $(X_{n\Delta t})_{n\in\mathbb{N}}$ from the available observations $(Y_{(n\Delta t, (n+1)\Delta t]})_{n\in\mathbb{N}}$.  To simplify notations, in this section we note $X_n$ instead of $X_{n\Delta t}$ and $Y_n$ instead of $Y_{((n-1)\Delta t, n\Delta t]}$.

Our aim in this section is to adapt the standard HMM algorithms to obtain estimations of the transition probabilities of the hidden chain $(X_n)$ from the observations $(Y_n)$. However, as explained in the previous section, $Y_{n}$ depends on both $X_{n-1}$ and $X_n$, therefore the emission scheme is non standard. Yet, note that the past $(Y_k)_{k\leq n-1}$ and current observations $Y_{n}$ are independent conditional on $X_{n-1}$. To circumvent the difficulty arising from the non standard emission scheme, we use the two-dimensional Markov chain corresponding to pairs of consecutive hidden states.

Denote by $(Z_n)_{n \in \mathbb{N}^*} =(X_{n-1},X_{n})_{n \in \mathbb{N}^*}$ the  two-dimensional hidden Markov chain with values in the state space $\mathbb{N}^2$. Then it is straightforward to prove that $(Z_n,Y_n)$ is a standard HMM with the following characteristics.
%
\begin{lemma}
\label{lem:HMMdef}
The process $(Z_n,Y_n)_{n\in\mathbb{N}^*}$ is a hidden Markov model with characteristics given by the triple $M = (Q,\psi,\rho)$, where\\
$\bullet$ the transition probability matrix $Q$ of the hidden chain $(Z_n)$ is 
\begin{align*}
Q_{(i,j),(i',j')}= \mathbb{P}(Z_{n+1}=(i',j')|Z_n=(i,j))=p_{i',j'}\delta_{i'=j},
\end{align*}
for  $(i,j), (i',j') \in \mathbb{N}^2$;\\
$\bullet$ the emission probability $\psi$ of process $Y$ given process $Z$ is
\begin{align*}
\psi_{(i,j)}(y)= \mathbb{P}(Y_n = y | Z_n=(i,j)) =  \frac{p_{i,(j, y)}}{p_{i,j}},
\end{align*}
for $i,j,y\in\mathbb{N}$, where $p_{i,(j, y)}=p_{i,(j, y)}(\Delta t)$ is characterized by Equation (\ref{eq:KolmoY});\\
$\bullet$ the initial distribution $\rho$ of the state process $Z$ is
\begin{align*}
\rho_{i,j}=\mathbb{P}(Z_1=(i,j))=\pi_i p_{i,j} ,
\end{align*}
for $i,j\in\mathbb{N}$, where $ \pi$ is the stationary distribution of $X$ given in Equation (\ref{eq6}).
\end{lemma}
%
The standard algorithm to estimate the most likely parameters $M = (Q,\psi,\rho)$ of the HMM given the observations is the forward-backward or Baum-Welch algorithm, which is a special case of the Expectation-Maximization (EM) algorithm, \cite{baum1966statistical}.
It is an iterative algorithm. At step $n$, given the current parameters $M^n = (Q^n,\psi^n,\rho^n)$, the likelihood of the observations $\mathbb{P}(Y_1=y_1,\ldots,Y_T=y_T | M^n)$ is maximized and the model $M^n$ is updated to $M^{n+1}$ accordingly. The main asset of the Baum-Welch algorithm is that the likelihood can be maximized explicitly. 

Unfortunately, running the standard HMM procedure from the R package HMM on $(Z_n,Y_n)$ yields numerical instability and errors due to the many zero entries in the transition matrix $Q$ of $Z$. Thus, we rewrote the main functions adapting the recursive formulas to the special sparse structure of $Q$ to obtain explicit estimators of the $p_{i,j}$.

\begin{theorem}
\label{th:HMMiter}
For the initial parameters $M^n = (Q^n,\psi^n,\rho^n)$, the maximum likelihood estimates of $M^{n+1} = (Q^{n+1},\psi^{n+1},\rho^{n+1})$ based on the observed tuple $(y_1,\ldots,y_T)$ is given by
\begin{align*}
Q_{(i,j),(i',j')}^{n+1} &= \dfrac{\sum_{t=1}^{T}\xi^n_{(i,j),(i',j')}(t)}{\sum_{t=1}^{T} \gamma^n_{(i,j)}(t) } \delta_{i'=j}, \\ 
\psi_{(i,j)}^{n+1}(y) &= \dfrac{\sum_{t=1}^{T} \mathbbm{1}_{y_t = y} \gamma^n_{(i,j)}(t)}{\sum_{t=1}^{T} \gamma^n_{(i,j)}(t)},\\
\rho_{i,j}^{n+1}&= \gamma^n_{(i,j)}(1),\\
\end{align*}
where
\begin{align*}
\gamma_{(i,j)}^n(t) &= \dfrac{\alpha_{(i,j)}^n(t) \beta_{(i,j)}^n(t)}{ \sum_{i,j\in\mathbb{N}} \alpha_{(i,j)}^n(t) \beta_{(i,j)}^n(t)},\\
\xi_{(i,j),(i',j')}^n(t+1)&= \dfrac{\alpha^n_{(i,j)}(t) p^n_{(i',j')}\psi^n_{(i',j')}(y_{t+1}) \beta^n_{(i',j')}(t+1) }{ \sum_{i,j\in\mathbb{N}} \alpha_{(i,j)}^n(t) \beta_{(i,j)}^n(t)}\delta_{i'=j},\\
p_{i,j}^{n}&= \dfrac{\sum_{i'\in\mathbb{N} }Q^{n}_{(i',i),(i,j)}}{\sum_{i'\in\mathbb{N}}\mathbbm{1}_{Q^{n}_{(i',i),(i,j)}\neq 0}},\\
\end{align*}
and the forward and backward sequences $\alpha^n$ and $\beta^n$ are defined recursively, for $2\leq t\leq T$, by
\begin{align*}
\begin{cases}
\alpha^n_{(i,j)}(1)= \rho^n_{(i,j)} \psi^n_{(i,j)}\left(y_{1}\right), \\
\alpha^n_{(i,j)}(t)=\psi^n_{(i,j)}\left(y_{t}\right) p^n_{i,j} \sum_{i'\in\mathbb{N}} \alpha^n_{(i',i)}(t-1),
\end{cases}\\
\begin{cases}
\beta^n_{(i,j)}(T)= 1, \\
\beta^n_{(i,j)}(t-1)=  \sum_{j'\in\mathbb{N}} p^n_{j,j'} \psi^n_{(j,j')}\left(y_{t}\right) \beta^n_{(j,j')}(t).
\end{cases}
\end{align*}
In addition, the estimated transition probabilities of $(X_n)$ are obtained by
\begin{align*}
p_{i,j}^{n+1}&= \dfrac{\sum_{i'\in\mathbb{N} }Q^{n+1}_{(i',i),(i,j)}}{\sum_{i'\in\mathbb{N}}\mathbbm{1}_{Q^{n+1}_{(i',i),(i,j)}\neq 0}}.
\end{align*}
\end{theorem}

Elements of proof are given in Section \ref{app:th3}.
After say $m$ iterations, the algorithm reaches the optimal parameters of the model and the estimates $\hat{\lambda}, \hat{\mu}$ and $\hat{\nu}$ are given by $$ (\hat{\lambda}, \hat{\mu},\hat{\nu})= g(p_{0,0}^{m},p_{0,1}^{m},p_{1,0}^{m}).$$
Note that the LBDI dynamics of $X$ together with the theoretical link between $X$ and $Y$ are only encoded in the initial model $M^0 = (Q^0,\psi^0,\rho^0)$ that is computed based on some triple $(\lambda^0,\mu^0,\nu^0)$ from Lemma \ref{lem:KolmoY} and Theorem \ref{main0}.
As usual for EM algorithm, the outcome is highly dependent on the initial condition. One must also tune the number $m$ of iterations and the way to truncate the infinite sums. All these points are discussed on synthetic data in Section \ref{sec:num}.
%
\section{Proofs}
\label{sec:proof}
This section is dedicated to the proofs of our main results.
\subsection{Proof of Lemma \ref{lem:KolmoY}}
\label{app:KolmoY}
%
\begin{proof}
We use the infinitesimal characterization of the distribution of the process.
For $h > 0$ and $j \leq i+y$, one has
\begin{align*} 
\lefteqn{p_{i,(j, y)}(t+h)}\\
&= \mathbb{P}\left(X_{t+h}=j ,Y_{ (0, t+h]}=y \mid X_{0}=i\right) \\ 
&= \mathbb{P}\left( X_{h}=i ,Y_{ (0, h]}=0,  X_{t+h}=j,Y_{ (h, t+h]}=y\mid X_{0}=i\right)  \\ 
&\quad + \mathbb{P}\left(X_{h}=i+1 ,Y_{ (0, h]}=0,  X_{t+h}=j,Y_{ (h, t+h]}=y \mid X_{0}=i\right) \\ 
&\quad +\mathbb{P}\left(X_{h}=i-1 ,Y_{ (0, h]}=1,  X_{t+h}=j,Y_{ (h, t+h]}=y-1   \mid X_{0}=i\right)+o(h) \\ 
&= \mathbb{P}\left( X_{h}=i ,Y_{ (0, h]}=0\mid X_{0}=i\right) \mathbb{P}\left( X_{t+h}=j,Y_{ (h, t+h]}=y\mid X_{h}=i\right)  \\ 
&\quad + \mathbb{P}\left(X_{h}=i+1 ,Y_{ (0, h]}=0 \mid X_{0}=i\right)  \mathbb{P}\left(X_{t+h}=j,Y_{ (h, t+h]}=y \mid X_{h}=i+1\right) \\ 
&\quad +\mathbb{P}\left(X_{h}=i-1 ,Y_{ (0, h]}=1\mid X_{0}=i\right)\mathbb{P}\left(X_{t+h}=j,Y_{ (h, t+h]}=y-1   \mid X_{h}=i-1\right)+o(h).
\end{align*}
Since $(X_t)$ is a homogeneous Markov process, and the jumps of $Y$ are directed by the jumps of $X$, we get 
\begin{align*} 
p_{i,(j, y)}(t+h)&= (1-((\lambda +\mu) i+\nu) h) p_{i,(j, y)}(t)\\
&\quad +(\lambda i+\nu) h p_{i+1,(j, y)}(t)+(\mu i) h p_{i-1,(j, -)}(t)+o(h).
\end{align*}
The result is obtained by subtracting $p_{i,(j, y)}(t)$ on both sides and letting $h$ go to zero.
\end{proof}
%
\subsection{Proof of Theorem \ref{main0}}
\label{app:th1}
%
To prove Theorem \ref{main0},  we update the method of characteristics, see \cite{fritzPartial}, to obtain a standard first-order linear partial differential equation for the probability generating function of $X_t$. Then, we derive an expression for the transition probabilities by using well-known results on series expansion. Although we made the assumption that $\lambda<\mu$, we give general formulas including the case $\lambda\geq \mu$ as they are interesting per se.
First, we need some technical results. 
%
\begin{lemma} \label{lemma1}
Let $i\in\mathbb{N}$ and $G(s,t)$ be a differentiable function on $[0, +\infty)^2$. The solution of the following linear partial differential equation 
\begin{align} \label{eq4}
 \left\{\begin{array}{ll}
 \dfrac{\partial G(s,t)}{\partial t} &=  \left( \mu - \lambda s \right) \left( 1-s \right) \dfrac{\partial G(s,t)}{\partial s} +  \nu(s-1) G(s,t), \\
  G(s,0) &=  s^i,
\end{array}\right.
\end{align}
is given by  
\begin{align*} 
G(s,t) =  \left\{\begin{array}{ll}
\dfrac{\left(\mu - \lambda \right)^{\frac{\nu}{\lambda}} \Big( s (\mu e^{(\lambda - \mu) t } - \lambda ) - \mu \left( e^{(\lambda - \mu) t } -1 \right)\Big)^{i}}{\Big(  \mu - \lambda  e^{(\lambda - \mu) t } + \lambda s \left( e^{(\lambda - \mu) t } -1 \right)  \Big)^{\frac{\nu}{\lambda} + i}} & \hspace*{1cm} \lambda \neq \mu, \\
\left( \dfrac{\lambda t + s \left(1- \lambda t \right)}{1+ \lambda t - \lambda ts}\right)^{i}\left(1+ \lambda t - \lambda ts  \right)^{-\frac{\nu}{\lambda}} & \hspace*{1cm}  \lambda = \mu.
\end{array}\right.
\end{align*}
\end{lemma}
%
\begin{proof}
To solve Equation (\ref{eq4}) we make use of above mentioned method of characteristics.\\
\textit{Step 1. Determining the Characteristics.} Consider the characteristic equation  
\begin{align*} 
\dfrac{d t}{d s}=-\dfrac{1}{(\mu-\lambda s)(1-s)}=\left\{\begin{array}{ll}
\dfrac{1}{\lambda-\mu}\left(\dfrac{1}{1-s}-\dfrac{\lambda}{\mu-\lambda s}\right)& \ \ \  \text { if }\ \ \lambda \neq \mu\\
-\dfrac{1}{\lambda(1-s)^{2}} & \ \ \  \text {if } \ \ \lambda=\mu.
\end{array}\right.
\end{align*}
By integrating with respect to $s$, we have  
\begin{align}  \label{eq19}
t = k - \left\{\begin{array}{ll}
\dfrac{1}{\lambda - \mu} \ln \left(\dfrac{1-s}{\mu -\lambda s }\right) & \ \ \ \text { if }\ \ \lambda \neq \mu\\
\dfrac{1}{\lambda\left( 1-s\right)} & \ \ \ \text { if } \  \ \lambda=\mu.
\end{array}\right.
\end{align}
Equation (\ref{eq19}) defines a family of characteristics parameterized by the integration constant $k$.\\
\textit{Step 2.} Along the characteristics in (\ref{eq19}),  $G$ satisfies 
\begin{align*} 
\dfrac{dG}{ds}= \dfrac{\nu G}{\mu -\lambda s}.
\end{align*}
Integrating with respect to $s$ yields
\begin{align*} 
G = C(k) \left(\mu - \lambda s\right)^{-\frac{\nu}{\lambda}},
\end{align*}
where $C(k)$ is another  integration constant, function of the  parameter k. Using Equation (\ref{eq19}), we obtain 
\begin{align}\label{eq20}
G(s,t) = \left\{\begin{array}{ll}
C\left(t + \dfrac{1}{\lambda - \mu} ln \left(\dfrac{1-s}{ \mu - \lambda s }\right)\right)\left(\mu - \lambda s\right)^{-\frac{\nu}{\lambda}}  & \ \ \ \text { if }\ \  \lambda \neq \mu \\
C\left(t+\dfrac{1}{\lambda\left(1-s\right)}\right) \left(\lambda \left(1- s\right)\right)^{-\frac{\nu}{\lambda}} & \ \ \  \text {if } \  \ \lambda=\mu.
\end{array}\right.
\end{align}
Now, we use the initial condition $G(s,0) = s^i$ to compute $C(k)$. 
Setting  $t = 0$ in equation (\ref{eq20}), we get 
\begin{align*}
s^i = \left\{\begin{array}{ll}
C\left( \dfrac{1}{\lambda - \mu} ln \left(\dfrac{1-s}{ \mu - \lambda s }\right)\right)\left(\mu - \lambda s\right)^{-\frac{\nu}{\lambda}}  & \ \ \ \text { if }\ \  \lambda \neq \mu \\
C\left(\dfrac{1}{\lambda\left(1-s\right)}\right) \left(\lambda \left(1- s\right)\right)^{-\frac{\nu}{\lambda}} & \ \ \  \text { if } \  \ \lambda=\mu,
\end{array}\right.
\end{align*}
so that one has
\begin{align*}
C(w) = \left\{\begin{array}{ll}
\left( \dfrac{e^{\left(\lambda - \mu \right)w} \mu -1}{e^{\left(\lambda - \mu \right)w} \lambda -1} \right)^i \left(\mu - \lambda \dfrac{e^{\left(\lambda - \mu \right)w} \mu -1}{e^{\left(\lambda - \mu \right)w} \lambda -1}\right)^{\frac{\nu}{\lambda}} & \ \ \ \text { if }\ \  \lambda \neq \mu \\
\left(1-\dfrac{1}{\lambda w} \right)^i w^{-\frac{\nu}{\lambda}} & \ \ \  \text { if } \  \ \lambda=\mu,
\end{array}\right.
\end{align*}
where 
\begin{align*}
w= \left\{\begin{array}{ll}
\dfrac{1}{\lambda - \mu} ln \left(\dfrac{1-s}{ \mu - \lambda s }\right)  & \ \ \ \text { if }\ \  \lambda \neq \mu \\
\dfrac{1}{\lambda\left(1-s\right)} & \ \ \  \text { if} \  \ \lambda=\mu.
\end{array}\right.
\end{align*}
Finally, we obtain  the required expression for $ \lambda \neq \mu$, 
\begin{align*}
 G_{i}(s,t) =  \dfrac{\left(\mu - \lambda \right)^{\frac{\nu}{\lambda}} \Big( s (\mu e^{(\lambda - \mu) t } - \lambda ) - \mu \left( e^{(\lambda - \mu) t } -1 \right)\Big)^{i}}{\Big(  \mu - \lambda  e^{(\lambda - \mu) t } + \lambda s \left( e^{(\lambda - \mu) t } -1 \right)  \Big)^{\frac{\nu}{\lambda} + i}},
 \end{align*}
and for $\lambda = \mu$ 
\begin{align*}
G_{i}(s,t) =   \left( \dfrac{\lambda t + s \left(1- \lambda t \right)}{1+ \lambda t - \lambda ts}\right)^{i}\left(1+ \lambda t - \lambda ts  \right)^{-\frac{\nu}{\lambda}},
 \end{align*}
hence the result.
\end{proof}

We can now turn to the proof of Theorem \ref{main0}.
%
\begin{proof} of Theorem \ref{main0}\\
Let us define the probability generating function of $X_t$: for any $(s,t ) \in \times [0, +\infty)^2$, and $ i \in \mathbb{N} $ set
\begin{align*}
G_{i}(s,t) = \sum_{j=0}^{+\infty} s^j p_{i, j}(t).
\end{align*}  
The forward Kolmogorov equation (\ref{eq1}) yields
\begin{align*}
\frac{\partial G_{i}(s,t)}{\partial t}
 &= \left( \lambda s^2 -\left( \lambda + \mu \right) s + \mu \right)\sum_{j=1}^{\infty} j s^{j-1} p_{i, j}(t) + \nu(s-1)  \sum_{j=0}^{\infty} s^j p_{i, j}(t),
\end{align*}  
so that the probability generating function satisfies the following  first-order linear partial differential equation :  
\begin{align*}
\frac{\partial G_{i}(s,t)}{\partial t} = \left( \mu - \lambda s \right) \left( 1-s \right) \frac{\partial G_{i}(s,t)}{\partial s} +  \nu(s-1) G_{i}(s,t),
\end{align*}  
with initial condition $G_i(s,0) = s^i.$
Now, using Lemma \ref{lemma1}, we have for $ \lambda \neq \mu$, 
\begin{align*}
G_{i}(s,t) =  \frac{\left(\mu - \lambda \right)^{\frac{\nu}{\lambda}} \Big( s (\mu e^{(\lambda - \mu) t } - \lambda ) - \mu \left( e^{(\lambda - \mu) t } -1 \right)\Big)^{i}}{\Big(  \mu - \lambda  e^{(\lambda - \mu) t } + \lambda s \left( e^{(\lambda - \mu) t } -1 \right)  \Big)^{\frac{\nu}{\lambda} + i}}\end{align*}
and for $\lambda = \mu$ 
\begin{align*}
G_{i}(s,t) =   \left( \frac{\lambda t + s \left(1- \lambda t \right)}{1+ \lambda t - \lambda ts}\right)^{i}\left(1+ \lambda t - \lambda ts  \right)^{-\frac{\nu}{\lambda}}.  
 \end{align*}
The case  $i = 0$ gives the following generating function of the transition probabilities $p_{0,j}(t)$, $t\geq 0$ 
\begin{align*}
G_0( s,t )  =  \left\{\begin{array}{ll}
 \Big(\frac{\mu - \lambda }{\mu - \lambda e^{(\lambda - \mu) t }} \Big)^{\frac{\nu}{\lambda}} \Big(1- \frac{ \lambda  ( 1-e^{(\lambda - \mu) t } )}{\mu - \lambda e^{(\lambda - \mu) t }} s  \Big)^{-\frac{\nu}{\lambda}}& \hspace*{1cm} \lambda \neq \mu, \\
\Big(\dfrac{1}{1+\lambda t}\Big)^{\frac{\nu}{\lambda}} \Big(1- \frac{\lambda t}{1+ \lambda t}s  \Big)^{-\frac{\nu}{\lambda}} & \hspace*{1cm}  \lambda = \mu.
\end{array}\right.
 \end{align*}
It is the generating function of the generalized negative binomial distribution with parameters 
\begin{align*}
r = \dfrac{\nu}{\lambda} \quad \textrm{and} \quad  q(t) = \left\{\begin{array}{ll}
\dfrac{\mu -\lambda}{ \mu - \lambda e^{(\lambda - \mu) t }} & \hspace*{1cm} \lambda \neq \mu, \\
 \dfrac{1}{1+\lambda t} & \hspace*{1cm}  \lambda = \mu.
\end{array}\right.
 \end{align*}
Therefore, for $j \in \mathbb{N}$ and $t \geq 0$,  one has
\begin{align}\label{eq80}
 p_{0,j}(t)= 
{j + r -1  \choose r -1}
q(t)^r (1-q(t))^j.
 \end{align}
The  general form  of $G_i$  for $i \not=0$ can be written as follows: 
\begin{align}\label{eq81}
 G_i(s,t)=G_0(s,t) N_i(s,t),
 \end{align}
where 
\begin{align*} N_i( s,t )  =  \left\{\begin{array}{ll}
 \Big(\dfrac{\mu - \lambda s - \mu(1-s) e^{(\lambda - \mu) t }}{\mu - \lambda s - \lambda(1-s) e^{(\lambda - \mu) t }} \Big)^i   & \hspace*{1cm} \lambda \neq \mu, \\
\left( \dfrac{1+\left(\lambda t -1 \right)\left(1-s \right)}{1+ \lambda t \left(1-s \right)}\right)^{i} & \hspace*{1cm}  \lambda = \mu,
\end{array}\right.
 \end{align*}
is the probability generating function of a linear  birth–death process without immigration. In essence, Equation (\ref{eq81}) reflects the independence between the process of contamination from the environment ($G_0$) and the self-contamination of the population ($N$). The linear birth and death process without immigration has the branching property that translates into the simple power form of its probability generating function, see for instance \cite{tavare2018linear} or \cite{bailey1991elements}. One can readily retrieve its transition probabilities $p^{(N)}_{i,j}(t), t\geq 0$ from the series expansion of $N_i(s, t)$:
\begin{align}\label{eq82}
p^{(N)}_{i,j}(t) = \sum_{l=0}^{\min (i, j)}
{i \choose l}
{i+j-l-1 \choose i-1}
\alpha(t)^{i-l} \beta(t)^{j-l}(1-\alpha(t)-\beta(t))^{l},
 \end{align}
where $\alpha(t) = \dfrac{\mu}{\lambda} (1-p(t))$ and $\beta(t) = 1-q(t)$.
Thus, the probabilities $p_{i,j}(t)$ are given by the following convolution  from Equation (\ref{eq81}):
\begin{align*} 
p_{i,j}(t) = \sum_{k=0}^{j} p_{0,k}(t) p^{(N)}_{i,j-k}(t),
 \end{align*}
and Equations (\ref{eq80}), (\ref{eq82}), finally yield the closed formulas
\begin{align*}
p_{i,j}(t) &= q(t)^r \sum_{l=0}^{min(i,j)} 
{i \choose l}
\left(\frac{\mu}{\lambda}\right)^{i-l} \left(1-q(t) \right)^{i+j-l}\left( \frac{1-(\frac{\mu}{\lambda}+1)(1-q(t))}{1-q(t)} \right)^{l}   \\ \nonumber
&\quad \times \sum_{k=0}^{j-l} 
{ r -1+k \choose  r -1}
{ i-1+j-l-k  \choose  i-1},
 \end{align*}
hence the result.
\end{proof}

From the previous proof, one can also retrieve the well known expression for the mean of the process $m(t)$
\begin{align*}
m(t) = \mathbb{E}\left[ X_{t}\mid X_{0} =i\right] = \left\{\begin{array}{ll}
  \frac{\nu}{ \lambda -\mu} \big( e^{(\lambda - \mu) t } -1 \big) +  i e^{(\lambda - \mu) t } & \hspace*{.5cm} \lambda \neq \mu, \\
 \nu t + i & \hspace*{.5cm}  \lambda = \mu,
\end{array}\right.
 \end{align*}
by derivating the probability generating function with respect to $s$ and taking $s=1$. 
The above result on the expectation is consistent with that obtained in \cite{ross2014introduction} using another method which consists in deriving an equation for $m(t + h)$, conditioning on $X_{t}$, and then solving a differential equation on $m(t)$ directly. 
%
\subsection{Proof of Theorem \ref{main1}}
\label{app:th2}
%
\begin{proof}
We make use of Equation (\ref{eq2}) for $(i,j)\in\{(0,0), (0,1), (1,0)\}$ and $t=\Delta t$ to obtain
\begin{align*}
p_{0,0}&= q^r,\\
p_{0,1}&= q^rr\left(1-q \right),\\
p_{1,0}&= q^r\left(\frac{\mu}{\lambda}\right)\left(1-q \right),
\end{align*}
with $r= \frac{\nu}{\lambda}$ and $q=q(\Delta t)=\frac{\mu -\lambda}{ \mu - \lambda e^{(\lambda - \mu) \Delta t}}$. 
The first $2$ equations above are equivalent to the system
\begin{align}\label{eq11}
\left\{\begin{array}{ll}
\ln(p_{0,0})  \ & = r\ln(q)  \\
\frac{p_{0,1}}{p_{0,0}}   \ & = r \left(1-q \right),
\end{array} \right.
\end{align}
which yields
\begin{align*}
\frac{p_{0,0}}{p_{0,1}} \ln(p_{0,0}) q e^{\frac{p_{0,0}}{p_{0,1}} \ln(p_{0,0}) q} = \frac{p_{0,0}^{\left(\frac{p_{0,0}}{p_{0,1}} +1 \right)} \ln(p_{0,0})}{p_{0,1}}, 
\end{align*}
so that  one has
\begin{align}\label{eq12}
q=\frac{p_{0,1}}{p_{0,0}(t) \ln(p_{0,0})} W\left( \frac{p_{0,0}^{\left(\frac{p_{0,0}}{p_{0,1}} +1 \right)} \ln(p_{0,0})}{p_{0,1}} \right),
\end{align}
 where $W$ is the Lambert function. Now, from Equations (\ref{eq11}) and (\ref{eq12}), we can easily find that
\begin{align*}
 r = \frac{\ln(p_{0,0})}{\ln(q)}.
\end{align*}
Now set
\begin{align*}
u= q e^{(\lambda-\mu)\Delta t}  = 1-\frac{p_{1,0}}{p_{0,0}}.
\end{align*}
Equation (\ref{eq2}) for $(i,j)\in\{(0,0), (0,1), (1,0)\}$ can now be rewritten as
\begin{align}\label{eq15}
 \left\{\begin{array}{ll}
 r  \ & = \frac{\nu}{\lambda} \\ 
 q  \ &  = \frac{\mu -\lambda}{ \mu - \lambda e^{(\lambda - \mu)\Delta t }} \\ 
 u \ & =  \frac{\left(\mu -\lambda \right)e^{(\lambda - \mu) \Delta t }}{ \mu - \lambda e^{(\lambda - \mu) \Delta t }}.\end{array} \right. 
\end{align}
Finally, solving System (\ref{eq15}) gives rise to the expressions of $\lambda,\mu$ and $\nu$ in function of the transition probabilities $p_{0,0},p_{0,1}$ and $ p_{1,0}$ given in Theorem \ref{main1}. Note that these expressions are unique so that the parameters are identifiable.
\end{proof}
%
\subsection{Proof of Theorem \ref{th:HMMiter}}
\label{app:th3}
%
\begin{proof}
Suppose the initial model is $M=(Q,\psi,\rho)$, and recall that $Z_n=(X_{n-1},X_n,Y_n)$ starts at $n=1$.
The forward probability $\alpha_{(i,j)}(t)$ is defined as the joint probability of observing the first $t$ values of $Y$ and being in state $(i,j)$ at time $t$. 
The Markov property directly yields the following recursion
$$
\begin{cases}
\alpha_{(i,j)}(1)=  \rho_{(i,j)}\psi_{(i,j)}\left(y_{1}\right),\\
\alpha_{(i,j)}(t)=\psi_{(i,j)}\left(y_{t}\right) p_{i,j} \sum_{i'\in\mathbb{N}} \alpha_{(i',i)}(t-1) \quad &  1 < t \leq T.
\end{cases}
$$
Similarly, the backward probability $\beta_{(i,j)}(t)$ is defined as the conditional probability of observing le last $T-t$ values of $Y$ after time $t$, given that the state at time $t$ is $(i,j)$. The Markov property also easily yields the recursive formula 
$$
\begin{cases}
\beta_{(i,j)}(T)= 1, \\
\beta_{(i,j)}(t)=  \sum_{j'\in\mathbb{N}} p_{j,j'} \psi_{(j,j')}\left(y_{t+1}\right) \beta_{(j,j')}(t+1) \quad &  1 \leq t < T.
\end{cases}
$$
The likelihood of the observations rewrites as follows by summing all the forward and backward products : $$ \mathbb{P}(Y_{1:t}= y_{1:t} | M ) = \sum_{i\in\mathbb{N}} \sum_{j\in\mathbb{N}} \alpha_{(i,j)}(t) \beta_{(i,j)}(t). $$

Now, let us consider the conditional probability of a transition from state $(i,j)$ at time $t$ to state $(i',j')$ at time $t+1$, given the observations $Y_{1:T} = y_{1:T}$. It can be expressed as
\begin{eqnarray}\nonumber
\xi_{(i,j),(i',j')}(t) &=& \mathbb{P}(Z_t=(i,j), Z_{t+1} = (i',j') |Y_{1:T} = y_{1:T}, (Q,\psi,\rho) ) \\ \nonumber 
&=& \dfrac{\alpha_{(i,j)}(t) p_{(i',j')}\psi_{(i',j')}(y_{t+1}) \beta_{(i',j')}(t+1) }{ \mathbb{P}(Y_{1:T} = y_{1:T} | (Q,\psi,\rho) )} \delta_{i'=j}. 
\end{eqnarray}
Denote by  $ \gamma_{(i,j)}(t) = \sum_{j'\in\mathbb{N}} \xi_{(i,j),(j,j')}(t)$ the conditional probability of being in state $(i,j)$ at time $t$, given the observations $Y_{1:T} = y_{1:T}$. One has
$$ \gamma_{(i,j)}(t) = \dfrac{\alpha_{(i,j)}(t) \beta_{(i,j)}(t)}{\mathbb{P}(Y_{1:T}= y_{1:T} | (Q,\psi,\rho) }.$$

Hence the updated estimation of the model is estimations by one step of the adapted Baum-Welch algorithm  is
\begin{eqnarray}\nonumber
Q_{(i,j),(i',j')}^{(n+1)} &=& \dfrac{\sum_{t=0}^{T-1}\xi_{(i,j),(i',j')}(t)}{\sum_{t=0}^{T-1} \gamma_{(i,j)}(t) } \delta_{i'=j}, \\ \nonumber
 \psi_{(i,j)}^{(n+1)}(y) &=& \dfrac{\sum_{t=0}^{T} \mathbbm{1}_{y_t = y} \gamma_{(i,j)}(t)}{\sum_{t=0}^{T} \gamma_{(i,j)}(t)},\\ \nonumber
 \rho_{i,j}^{(n+1)}&=& \gamma_{(i,j)}(t),
\end{eqnarray}
as expected.
\end{proof}
%
\section{Numerical study}
\label{sec:num}
%
We now investigate the numerical performance of our estimation procedure both on synthetic and real data.
All codes were developed in R on a laptop with a $1.10$ GHz processor. To implement the Lambert-W function, we use the lambertW R package \cite{Goerg}. For Baum-Welch algorithm we adapted functions from the package HMM \cite{himmelmann2022package}.
%
\subsection{Synthetic data}
\label{sec:num-simu}
We first explore the performance of the estimators $ \hat{\lambda}^n,\hat{\mu}^n$ and $\hat{\nu}^n$ on simulated data, and investigate the sensitivity of the estimation to changes in the main hyper-parameters of the procedure. In order to quantify the loss of information due to the hidden information, we start with the setting where $(X_n)$ is observed, then turn to the HMM framework.
%
\subsubsection{Observations of infected counts}
\label{sec:num-simu-X}
%
When the discrete time Markov chain $(X_n)$ is observed, one can use the maximum likelihood estimates for the transition probabilities ${p}_{i,j}$, see e.g. \cite{athreya1992bootstrapping}
\begin{align*}
\hat{p}^n_{i,j} = \frac{\sum_{k=0}^{n-1} \mathbbm{1}_{\{X_{k} =i, X_{k+1}=j\}}}{\sum_{k=0}^{n-1} \mathbbm{1}_{\{X_k =i\}}}.
\end{align*}
%
We simulated a LBDI process starting from $X_{0}=0$ 
with  $\lambda=0.03 $, $ \mu=0.1 $ and $\nu=0.01$, until a (continuous) time horizon $H$ for several values of $\Delta t$. To estimate the transition probabilities between discrete states by the maximum likelihood method (MLE), we use the markovchainFit function of the markovchain package for discrete-time Markov chain models \cite{spedicato2016markovchain}. 

Table \ref{table1} presents the estimated values of the parameters with the asymptotic standard error (asymptotic standard deviation divided by $\sqrt{n}$) calculated from Corollary \ref{main2} with matrix $\Sigma'$ from Equation (\ref{eq=sigma'}). Figure \ref{fig:variance} shows the asymptotic standard deviations for all $3$ estimators in function of the time lapse $\Delta t$.
For convenience, the numerical values of the asymptotic standard deviations are also provided in Table~\ref{table-sigma}.
%
\begin{table}[hpt]
\caption{Estimated values and standard error when the infected counts are observed with period $\Delta t$ and time horizon $H$ (synthetic data, infected observed, true parameter values $\lambda=0.03 $, $ \mu=0.1 $ and $\nu=0.01$)}
\label{table1}
\begin{center}
\begin{tabular}{@{}cccc@{}}
\toprule%
& \multicolumn{3}{@{}c@{}}{$H=1000$}  \\
\cmidrule{2-4} %
Time lapse  $\Delta t$& $1$ & $7$ & $30$ \\
Number of observations $n$ & $1001$&$143$&$34$ \\
\midrule
$\hat\lambda^n$  & 0.027 (0.100)  & 0.081 (0.053)& 0.049 (0.065) \\
$\hat\mu^n$  & 0.084 (0.035)  & 0.105 (0.046) & 0.070 (0.128)  \\
$\hat\nu^n$ & 0.012 (0.484)  & 0.007 (0.106) & 0.007 (0.076)  \\
\hline
\end{tabular}
\begin{tabular}{@{}cccc@{}}
\toprule%
&\multicolumn{3}{@{}c@{}}{$H=5000$}   \\
\cmidrule{2-4} %
Time lapse  $\Delta t$& $1$ & $7$ & $30$ \\
Number of observations $n$ & $5001$&$715$&$167$\\
\midrule
$\hat\lambda^n$  & 0.041 (0.045)  & 0.052 (0.024)&0.064 (0.029)\\
$\hat\mu^n$  &0.088 (0.016) & 0.132 (0.020) & 0.117 (0.058)\\
$\hat\nu^n$ & 0.009 (0.216) & 0.012 (0.047) & 0.009 (0.034) \\
\hline
\end{tabular}
\begin{tabular}{@{}cccc@{}}
\toprule%
&\multicolumn{3}{@{}c@{}}{$H=50\ 000$}  \\
\cmidrule{2-4} %
Time lapse  $\Delta t$& $1$ & $7$ & $30$ \\
Number of observations $n$ &$50001$&$7143$&$1667$\\
\midrule
$\hat\lambda^n$  &0.030 (0.014)  &0.023 (0.007)& 0.020 (0.009)\\
$\hat\mu^n$  &  0.098 (0.004)  & 0.097 (0.006)&0.102 (0.018) \\
$\hat\nu^n$ & 0.010 (0.068)  & 0.010 (0.014)& 0.010 (0.011)  \\
\hline
\end{tabular}
\end{center}
\end{table}
%
\begin{figure}[hpt]
    \centering
    \includegraphics[scale=0.8]{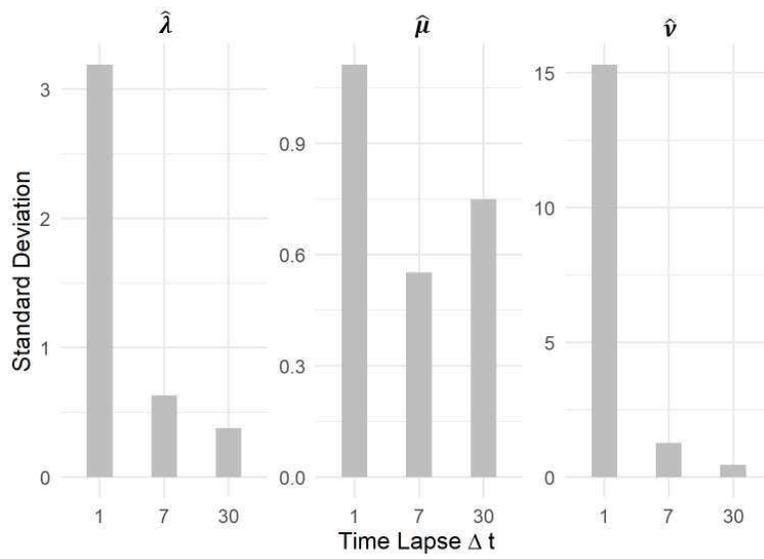} 
    \caption{Asymptotic standard deviations as a function of the time lapse $\Delta t$ for estimators $\hat \lambda$, $\hat\mu$ and $\hat\nu$ obtained from the maximum likelihood estimates of the transition probabilities (synthetic data, infected observed, true parameter values $\lambda=0.03 $, $ \mu=0.1 $ and $\nu=0.01$).}
    \label{fig:variance}
\end{figure}
%

Table \ref{table1} shows the convergence of the estimators as the number of observations increases. A high number of observations is required to capture the right order of magnitude. This confirms that even in a complete observation framework, retrieving the parameters from discrete observations of the process is demanding.

Figure \ref{fig:variance} illustrates that the asymptotic variance of the estimators strongly depends on $\Delta t$, and increases as $\Delta t$ decreases, with a different magnitude for the three coefficients, the impact being stronger on the third coefficient $\nu$ (immigration). This suggest a possible optimal choice to be made between high frequency observations with high variance versus low frequency ones with smaller variance. However in medical data collection practice the  sampling frequency may not be chosen by the statistician.
%
\subsubsection{Observations of cumulated new isolated counts}
\label{sec:num-simu-Y}
%
We now turn to the more realistic case when only the cumulated new isolated counts $(Y_n)$ are observed. 
We simulated a single trajectory of a LBDI process starting from its invariant distribution $X_{0}\sim\pi$ with the same parameters as in the previous section $\lambda=0.03 $, $ \mu=0.1 $ and $\nu=0.01$. Then we extracted the cumulated new infected counts for a fixed value of $\Delta t = 1$, see Fig. \ref{fig:my_label2} and ran the HMM procedure described in Section \ref{sec:HMM}. This was repeated $100$ times in order to capture the empirical estimation error. Additional simulation results are also provided in the next Section \ref{sec:dataARS} with a similar setting but a different set of parameters.

Figure \ref{fig:my_label1+2} displays a sample trajectory for these parameters, with the hidden values of the infected counts on top (Fig. \ref{fig:my_label1}) and the observed cumulated declared cases on the bottom (Fig \ref{fig:my_label2}).
%
\begin{figure}[hpt]
    \centering
    \begin{subfigure}{.95\textwidth}
    \includegraphics[scale=0.4]{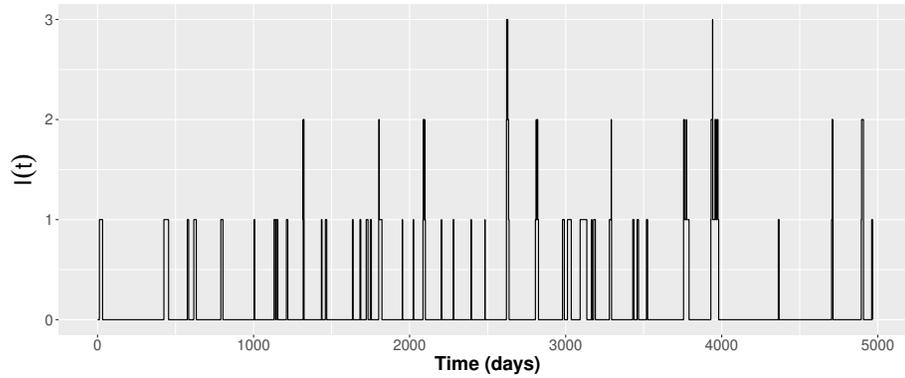} 
    \caption{Path of the birth-death process with immigration X(t)}
    \label{fig:my_label1}
    \end{subfigure}
    \begin{subfigure}{.95\textwidth}
    \includegraphics[scale=0.4]{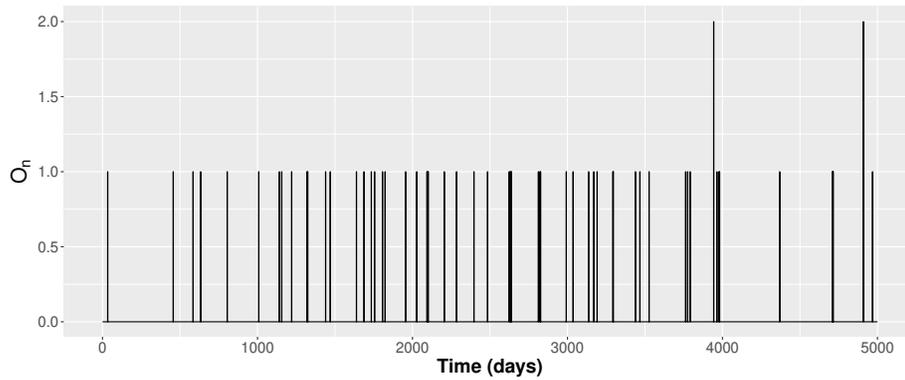} 
    \caption{The cumulated new isolated $Y_n$ corresponding to process $X$.}
    \label{fig:my_label2}
    \end{subfigure}
    \caption{Example of a sample path of the LBDI process in continuous time and the corresponding cumulated new isolated counts for a time lapse $\Delta t=1$ until a time horizon $H=5000$ with parameters $\lambda=0.03 $, $ \mu=0.1 $ and $\nu=0.01$ (synthetic data)}
  \label{fig:my_label1+2}  
\end{figure}

\paragraph{Tuning of the HMM} 
For the HMM procedure, we chose the initial parameters in the following ranges $\lambda^{(0)}\in[0.028,0.034], \mu^{(0)}\in[0.08,0.15] $ and $ \nu^{(0)}\in[0.008,0.015]$, the  maximum number of iterations is set to $500$ and the stopping criterion (norm of the difference between the current estimate and the previous one) to $10^{-9}$. We executed the HMM algorithm for each value of the input triple $\lambda^{(0)},\mu^{(0)}$ and $\nu^{(0)}$, and obtained output triples together with the (truncated) likelihood of the corresponding model. Recall that the likelihood of the model is an output of the algorithm as it is given by
\begin{align*}
\sum_{i,j} \alpha^n_{(i,j)}(T). 
\end{align*}
We then selected the best output triple as the one corresponding to the model with the highest likelihood
\begin{align*}
(\hat{\lambda}, \hat{\mu},\hat{\nu})= \argmax_{(\hat{\lambda}^n, \hat{\mu}^n,\hat{\nu}^n)} \mathbb{P}( Y_{1:T} |M^n). 
\end{align*}

\paragraph{Impact of the truncation parameter} 
Numerically, one cannot deal with infinite state spaces. Therefore, we truncated the transition matrix $P=(p_{i,j})$ on $\mathbb{N}^2$ to matrix $\bar P=(\bar p_{i,j})$ on the finite state space $\{0, 1, \ldots, N\}$ as follows. We calculated the coefficients $\left(p_{i,j}, 0 \leq i \leq N, 0 \leq j \leq N-1 \right) $ of the transition matrix $P$ with Equation (\ref{eq2}) and set
\begin{align*}
\bar p_{i,j}& = p_{i,j} \quad \quad \text{for} \ 0 \leq i \leq N, 0 \leq j \leq N-1,\\
\bar p_{i,N}&= 1 -\sum_{j=0}^{N-1} p_{i,j} \quad \quad \text{for} \ i=0,\ldots, N.
\end{align*}
so that $\bar P$ is still a transition matrix. This roughly corresponds to aggregating all states above $N$.
As we are only interested in the transition probabilities from $0$ to $0$, $0$ to $1$ and $1$ to $0$, one can expect that this truncation procedure should not have a big impact on the estimation. The sensitivity of the output to this truncation parameter is i displayed on Figure \ref{fig:N}, and detailed values are provided in Table \ref{table2}.
%
\begin{figure}[hpt]
    \centering
    \includegraphics[scale=0.8]{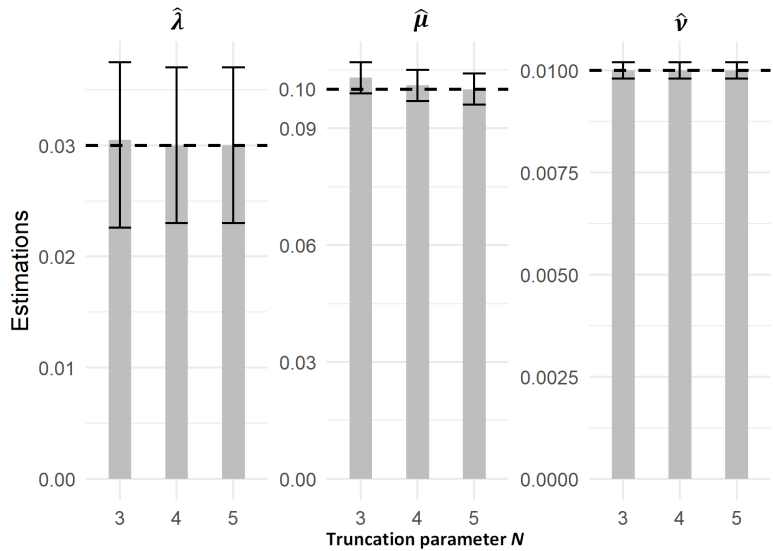} 
    \caption{Impact of the choice of the truncation parameter $N$ on the estimation with empirical $95\%$ confidence intervals ($100$ replications, synthetic data, infected hidden, true parameter values $\lambda=0.03 $, $ \mu=0.1 $ and $\nu=0.01$)}
    \label{fig:N}
\end{figure}

The reported values suggest that the choice of $N$ has a limited effect on the parameter estimates, with relatively similar values observed across different values of $N$, both in terms of estimated value and variance.  In terms of complexity, increasing the size of the truncation parameter $N$, significantly increases  the execution time of the EM algorithm, as expected, as the number of parameters to estimate is proportional to $(N+1)^2$. 

\paragraph{Impact of the time horizon} 
We set $N=5$ and we illustrate the impact of the time horizon $H$ on our estimation scheme on Figure \ref{fig:nobs}, see also Table \ref{table3} for the numerical values. As expected, the accuracy of the estimates shows gradual enhancement with a larger number of observations, accompanied by a reduction in the empirical standard error.
%
\begin{figure}[hpt]
    \centering
    \includegraphics[scale=0.8]{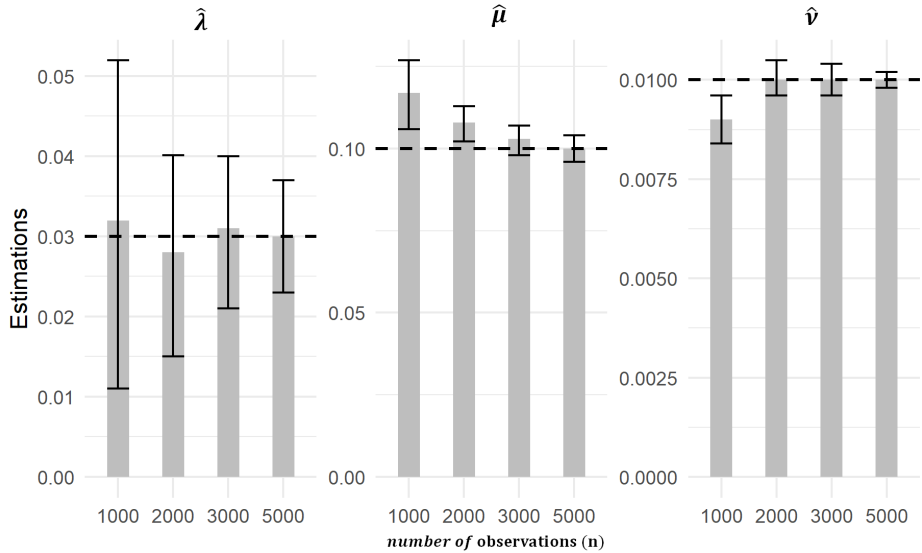} 
    \caption{Impact of the number of observations $n$ on the estimation on the estimation for a truncation parameter $N=5$ with empirical $95\%$ confidence intervals ($100$ replications, synthetic data, infected hidden, true parameter values $\lambda=0.03 $, $ \mu=0.1 $ and $\nu=0.01$)}
    \label{fig:nobs}
\end{figure}
%
As expected, the quality of the estimation improves slowly with the number of observations. 
%
\subsubsection{Comparison with least-squares estimation}
\label{sec:ls}
%
As an additional investigation, we tried to use more information from the output transition matrix from the HMM algorithm, rather that only the top left $3$ transition probabilities. More precisely, instead of using an analytic inversion formula, we used a numeric least squares estimation method. We minimized the sum of squares of the differences between the entries of the output transition matrix and their theoretical expression for a LBDI process. New estimators are thus obtained as follows
%
\begin{equation}\label{eq::1}
(\hat\lambda^n, \hat\mu^n,\hat\nu^n)=\underset{\lambda, \mu, \nu}{\operatorname{argmin}} \sum_{i, j}^{N-1} \left(p_{i j}(\Delta t ; \lambda, \mu, \nu)-\widehat{p}_{i j}^n\right)^2.
\end{equation}
%
Note that we did not use the last row and column of the transition matrix as they are likely to have a biased expression due to the truncation.

To achieve the estimation of these parameters, we employ the Limited-memory Broyden-Fletcher-Goldfarb-Shanno optimization method  implemented in the R programming environment \cite{coppolalbfgs}. 
The outcomes and their corresponding standard errors ($100$ replications) are presented on Figure \ref{fig:least-squares} 
\begin{figure}[hpt]
    \centering
    \includegraphics[scale=0.8]{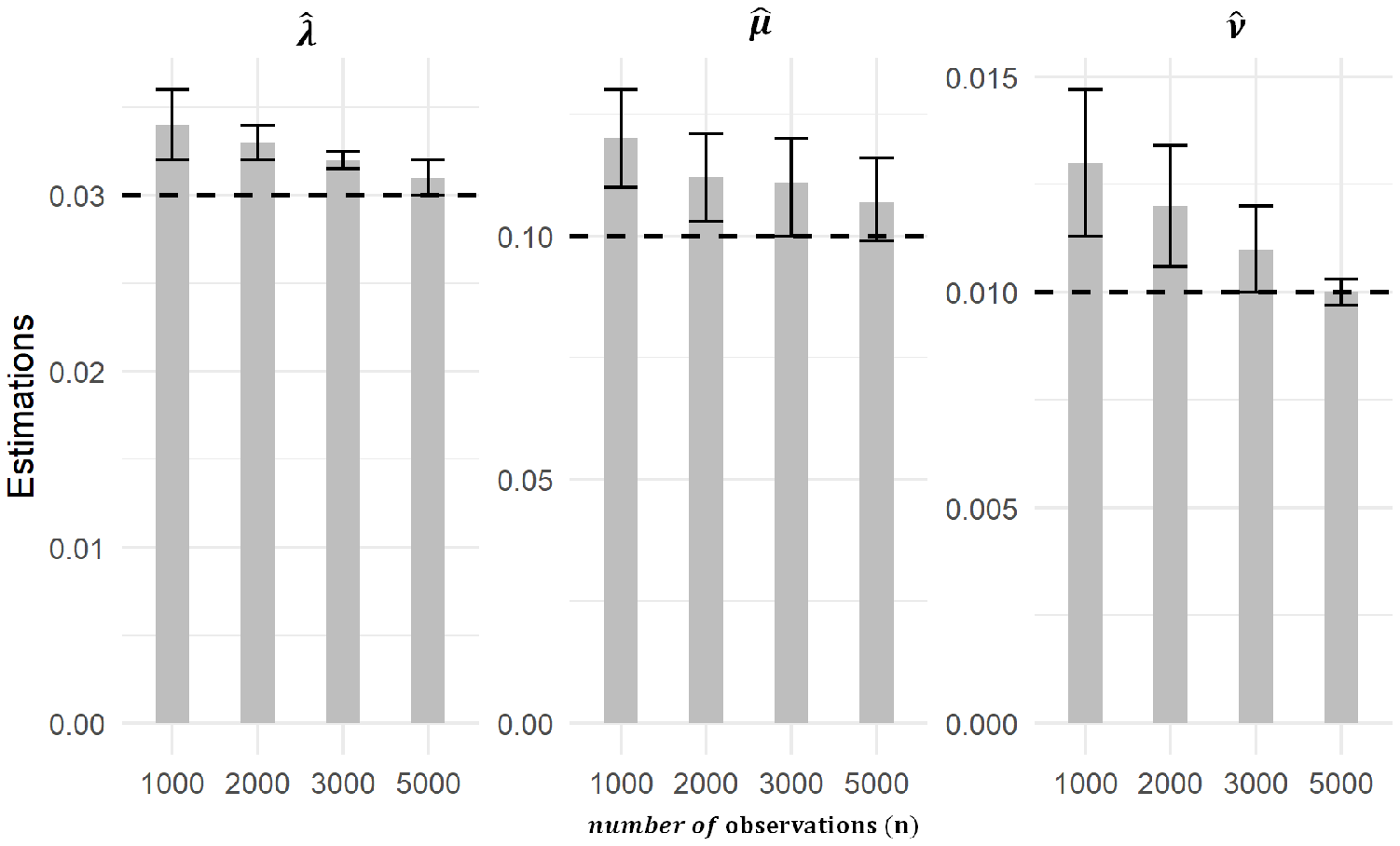} 
    \caption{Impact of the number of observations $n$ on parameter estimations and their confidence intervals using HMM with least-squares approximation ($N=5$, $100$ replications, synthetic data, infected hidden, true parameter values $\lambda=0.03 $, $ \mu=0.1 $ and $\nu=0.01$))}
        \label{fig:least-squares}
\end{figure}
%
(numerical values displayed in Tables \ref{table5.1}) for a  varying numbers of observations $n$ and on Figure \ref{fig:least-squares2} (numerical values displayed in Tables \ref{table:5.2}) for varying hidden state space truncation parameter $N$.
\begin{figure}[hpt]
    \centering
    \includegraphics[scale=0.8]{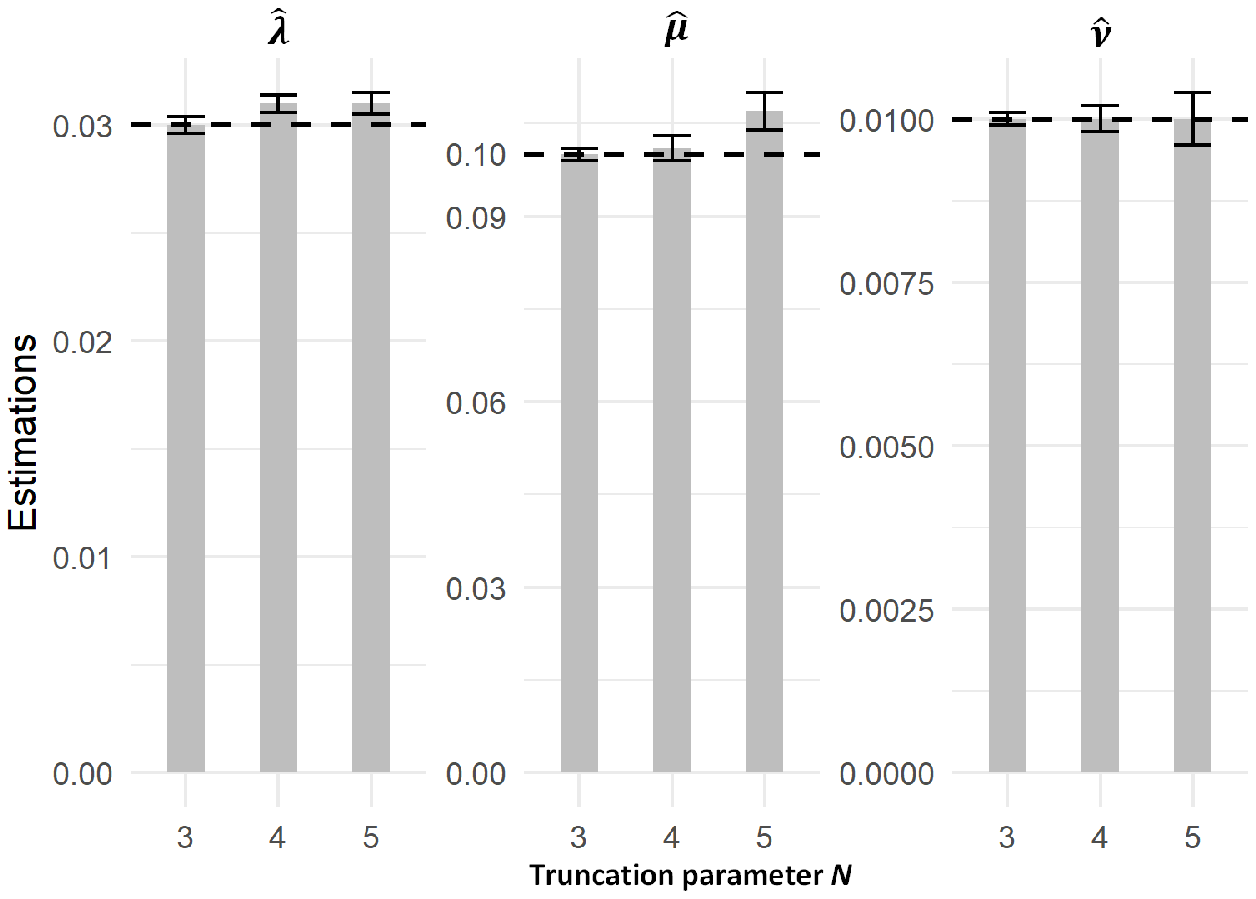} 
    \caption{Impact of the truncation parameter $N$ on parameter estimations and their confidence intervals using HMM with least-squares approximation ($H=5000$, $100$ replications, synthetic data, infected hidden, true parameter values $\lambda=0.03 $, $ \mu=0.1 $ and $\nu=0.01$)}
        \label{fig:least-squares2}
\end{figure}

The estimation method based on the combination of the HMM procedure and least-squares fitting exhibits significant sensitivity to truncation, that seems to deteriorate when $N$ grows. This could be explained by the fact that the output transition matrix from the HMM procedure is not constrained to fit a LBDI model. Therefore taking into account more parameters may make the deviation more obvious. 

The confidence intervals are significantly smaller than for analytical inversion, but the estimation is biased especially for a low number of observations. To investigate further the bias-variance compromise, we computed the mean squared error for both approaches. Results are displayed on Table \ref{tab:MSE-HMM-ls}.
%
\begin{table}[hpt]
\caption{Mean squared error using HMM with analytic inversion of the top left 3 transition probabilities (MSE a) and least-squares fitting (MSE ls) (synthetic data, infected hidden, $N=5$, $100$ replications, true parameter values $\lambda=0.03 $, $ \mu=0.1 $ and $\nu=0.01$)}
\label{tab:MSE-HMM-ls}
\begin{center}
\begin{tabular}{@{}ccccc@{}}
\toprule%
& \multicolumn{4}{@{}c@{}}{$\lambda$}  \\
\cmidrule{2-5} %
time horizon $H$&$1000$&$2000$&$3000$&$5000$ \\
number of observations $n$&$1001$&$2001$&$3001$& $5001$ \\
\midrule
MSE a   & $9.10^{-3} $ & $5.10^{-3}$& $3.10^{-3}$ & $\phantom{2.}10^{-3}$\\
MSE ls & $9.10^{-5}$  &$ 4.10^{-5}$ & $3.10^{-5}$&$2.10^{-5}$\\
\hline
\end{tabular}
\begin{tabular}{@{}ccccc@{}}
\toprule%
& \multicolumn{4}{@{}c@{}}{$\mu$}  \\
\cmidrule{2-5} %
time horizon $H$&$1000$&$2000$&$3000$&$5000$ \\
number of observations $n$&$1001$&$2001$&$3001$& $5001$ \\
\midrule
MSE a   & $2.10^{-3} $ & $9.10^{-4}$& $5.10^{-4}$ & $3.10^{-4}$\\
MSE ls & $3.10^{-3} $  &$ \phantom{9.}10^{-3}$ & $\phantom{5.}10^{-3}$&$7.10^{-4}$\\
\hline
\end{tabular}
\begin{tabular}{@{}ccccc@{}}
\toprule%
& \multicolumn{4}{@{}c@{}}{$\nu$}  \\
\cmidrule{2-5} %
time horizon $H$&$1000$&$2000$&$3000$&$5000$ \\
number of observations $n$&$1001$&$2001$&$3001$& $5001$ \\
\midrule
MSE a   & $\phantom{9.}10^{-5} $ & $5.10^{-6}$& $3.10^{-6}$ & $2.10^{-6}$\\
MSE ls & $9.10^{-5} $  &$ 5.10^{-5}$ & $3.10^{-5}$&$\phantom{2.}10^{-5}$\\
\hline
\end{tabular}
\end{center}
\end{table}
%
One can notice the the two approaches perform differently for the 3 parameters. For a number of observations around $2000$, the least-squares approach is more precise (by 2 orders of magnitude) than the analytical inversion for the intrinsic contamination parameter $\lambda$, but worse (by 1 order of magnitude) for the extrinsic contamination parameter $\nu$ and the declaration rate $\mu$. The difference seems to vanish for $\mu$ as the number of observations increases, but seems to remain of the same order of magnitude for the other two parameters.

\subsection{Typhoid data in Mayotte}
\label{sec:dataARS}
%
We now investigate a real data set provided by the ARS of Mayotte. It contains the daily counts of new declared cases of typhoid fever on the island between 2018 and 2022. This is the data set that motivated our study. 
We have $1816$ observations taking values between $0$ and $4$ with a total of $299$ recorded cases. The average number of declared cases is $59.8$ cases per year. Noteworthy is the observation that the year 2022 was characterized by a significantly more severe epidemic, with $100$ reported cases, see Fig. \ref{fig:my_label3} (where weekly counts are displayed) and Table \ref{table4}.
%
\begin{figure}[hpt]
    \centering
    \includegraphics[scale=0.4]{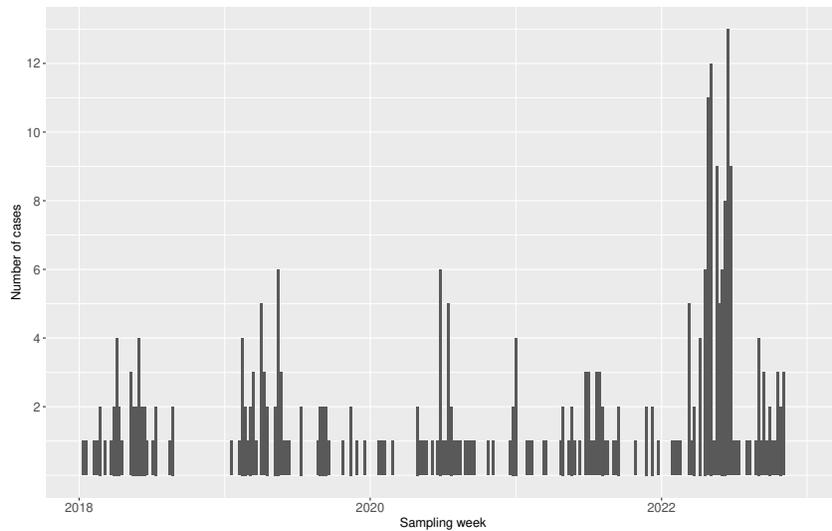} 
    \caption{Number of weekly cumulated new declared cases of typhoid fever in Mayotte between 2018 and 2022.}
    \label{fig:my_label3}
\end{figure}
%
\begin{table}[hpt]
\caption{Number of daily cumulated new declared cases of typhoid fever in Mayotte between 2018 and 2022}
\label{table4}
\begin{center}
\begin{tabular}{@{}cccccc@{}}
\toprule
observed value &$0$&$1$&$2$&$3$&$4$ \\
\midrule
number of observations &$1596$&$167$&$34$& $12$&$7$ \\
\hline
\end{tabular}
\end{center}
\end{table}
%
 The transitions for $0$ to $0$, $0$ to $1$ and $1$ to $0$ are the most frequently observed, constituting a cumulated proportion of $94.27\%$, see Table \ref{tab:Mayotte}. This guided our decision to consider only these transitions for the estimation of model parameters.
%
\begin{table}[hpt]
\caption{Proportions of observed daily transitions in declared cases of typhoid fever in Mayotte between 2018 and 2022}
\label{tab:Mayotte}
\begin{center}
\begin{tabular}{@{}ccccccc@{}}
\toprule
transition &$0\rightarrow0$ & $0\rightarrow1$ & $1\rightarrow0$ & $1\rightarrow1$ & $1\rightarrow2$ & other \\
\midrule
proportion &$\mathbf{79.62\%}$ & $\mathbf{7.05\%}$ & $\mathbf{7.60\%}$ & $1.15\%$ & $0.33\%$  & $4.25\%$\\
\hline
\end{tabular}
\end{center}
\end{table}

A stochastic model for the propagation of typhoid fever in Mayotte is well adapted to the low prevalence figures. Moreover, there are long lapses without cases followed by spontaneous surges of new cases, indicating the existence of an exogenous source of contamination. The simplest model corresponding to this case is thus the LBDI model under positive recurrent regime. As Typhoid fever is a notifiable disease in Mayotte, we consider that all cases are observed.
%
\subsubsection{Parameter estimation}
\label{sec:dataARS1}
%
Based on literature \cite{mushayabasa2011impact,pitzer2014predicting,pitzer2015mathematical,abboubakar2021mathematical}, the expert advice of typhoid specialists in Mayotte, and our synthetic data exploration, the stopping criterion for the HMM was set at $10^{-9}$, with a maximum number of iterations limited to $500$, and the truncation parameter was fixed at $N=4$. We chose initial parameters in the following ranges  $\lambda^{(0)}\in[0.05,0.08]$,  $\mu^{(0)}\in[0.11,0.25]$ and $\nu^{(0)}\in[0.015,0.03]$. This allows to initialize $P^{0}$, $Q^{0}$, $\psi^{0}$ and $\rho^0$. 

We selected the triple with the highest likelihood, namely
\begin{align*}
 \hat{\lambda}=0.054,  \quad \hat{\mu}=0.132 \quad \textrm{and} \quad  \hat{\nu}= 0.017,
\end{align*}
see Table \ref{table5} for complete results.

However, it is crucial to acknowledge the inherent precision limitations in these estimates, primarily due to the relatively modest scale of Mayotte dataset, comprising less than $2000$ observations. As the estimated parameters are significantly different from those we tested in the previous section, we conducted supplementary investigation on synthetic data in order to evaluated the estimation error in this case.
%
\subsubsection{Investigation of the estimation error}
\label{sec:dataARS2}
%
First, we performed simulations for $100$ trajectories with horizon $H=2000$ from a LBDI model with parameter values equal to our estimates $\lambda=0.054$, $\mu=0.132$, $\nu= 0.017$. Results are presented in Table \ref{table::5}.
%
\begin{table}[hpt]
\caption{Estimates $\hat\lambda^n,\hat\mu^{n}$ and $\hat\nu^{n}$  and $95\%$ confidence intervals for $100$ Trajectories with $2000$ observations (synthetic data, infected hidden, true parameter values $\lambda=0.054 $, $ \mu=0.132 $ and $\nu=0.017$)}
\label{table::5}
\begin{center}
\begin{tabular}{@{}cccc@{}}
\toprule
Parameter & Estimations & Confidence interval & mean squared error\\ 
\midrule
$\lambda$  &$0.066$   & $\ [0.051;0.081]\ $    & $8.10^{-3}$\\
$\mu$  & $ 0.130$   &  $\ [0.126;0.134]\ $&  \ \ \ $10^{-3}$  \\
$\nu$ &  $ 0.017$     & $[0.0161;0.0173]$ & $5.10^{-5}$ \\
\hline
\end{tabular}
\end{center}
\end{table}
%
As expected form our original study on synthetic data, the estimation error is quite high for this data set size. It is slightly better for the declaration delay $\mu$ than for the other two contamination parameters $\lambda$ (intrinsic contamination) and $\nu$ (extrinsic contamination). Note that given this level of error, the positive recurrence condition $\lambda<\mu$ still seems valid.
%
\section{Conclusion}
\label{sec:ccl}
%
We have proposed a methodology to estimate the unknown parameters of a LBDI process in the mathematically unusual but epidemiologically common case where only cumulated new isolated counts over given periods of time are available. Our approach relies on both analytical formulas and numerical iterations. We investigated the properties of the estimators on synthetic data and applied our procedure to a real data set of typhoid fever in Mayotte that motivated this work.

Typhoid fever is an endemic risk disease in the territory of Mayotte, as evidenced by the epidemic observed in 2022 in the north of the island. Due to the number of observed cases and the limited hospitalization capacity, it is constantly monitored by the DéSUS team of ARS Mayotte. The strengthening of the internal crisis management strategy and the provision of innovative tools to monitor epidemics remains a priority for ARS Mayotte. This approach is a first step toward the quantitative understanding of the propagation of typhoid or other similar water-borne diseases in Mayotte. We proposed a simplified model with few interpretable parameters. It would be interesting to split the data set given locally meaningful covariates such as easy access to potable water to investigate wether the coefficients guiding the propagation of the disease differ with the sanitary conditions of the population, to include a spatial propagation component, or to investigate forecasting aspects to try to characterize the risk of outbreak of an epidemic, to name just a few possible new directions with possibly high public health impact. 
\bibliography{mybib}

\begin{thebibliography}{10}
\expandafter\ifx\csname url\endcsname\relax
  \def\url#1{\texttt{#1}}\fi
\expandafter\ifx\csname urlprefix\endcsname\relax\def\urlprefix{URL }\fi
\expandafter\ifx\csname href\endcsname\relax
  \def\href#1#2{#2} \def\path#1{#1}\fi

\bibitem{ijerph7103657}
J.~P.~S. Cabral, Water microbiology. bacterial pathogens and water,
  International Journal of Environmental Research and Public Health 7~(10)
  (2010) 3657--3703.

\bibitem{ASHBOLT2004229}
N.~J. Ashbolt, Microbial contamination of drinking water and disease outcomes
  in developing regions, Toxicology 198~(1-3) (2004) 229--238.

\bibitem{Salmonella}
S.-K. Eng, P.~Pusparajah, N.-S. Ab~Mutalib, H.-L. Ser, K.-G. Chan, L.-H. Lee,
  Salmonella: a review on pathogenesis, epidemiology and antibiotic resistance,
  Frontiers in Life Science 8~(3) (2015) 284--293.

\bibitem{Salmonellabis}
A.~Karkey, T.~Jombart, A.~W. Walker, C.~N. Thompson, A.~Torres, S.~Dongol,
  N.~Tran Vu~Thieu, D.~Pham~Thanh, D.~Tran Thi~Ngoc, P.~Voong~Vinh, A.~C.
  Singer, J.~Parkhill, G.~Thwaites, B.~Basnyat, N.~Ferguson, S.~Baker, The
  ecological dynamics of fecal contamination and salmonella typhi and
  salmonella paratyphi a in municipal kathmandu drinking water, PLOS Neglected
  Tropical Diseases 10 (2016) 1--18.

\bibitem{edelman1986summary}
R.~Edelman, M.~M. Levine, Summary of an international workshop on typhoid
  fever, Reviews of infectious diseases 8~(3) (1986) 329--349.

\bibitem{lauria2009optimization}
D.~T. Lauria, B.~Maskery, C.~Poulos, D.~Whittington, An optimization model for
  reducing typhoid cases in developing countries without increasing public
  spending, Vaccine 27~(10) (2009) 1609--1621.

\bibitem{mushayabasa2011impact}
S.~Mushayabasa, Impact of vaccines of controlling typhoid fever in
  kassana-nankana district of upper east region of ghana: insight from a
  mathematical model, Journal of Modern Mathematics and Statistics 5 (2011)
  54--59.

\bibitem{pitzer2014predicting}
V.~E. Pitzer, C.~C. Bowles, S.~Baker, G.~Kang, V.~Balaji, J.~J. Farrar, B.~T.
  Grenfell, Predicting the impact of vaccination on the transmission dynamics
  of typhoid in south asia: a mathematical modeling study, PLoS neglected
  tropical diseases 8~(1) (2014) e2642.

\bibitem{cvjetanovic1971epidemiological}
B.~Cvjetanovi{\'c}, B.~Grab, K.~Uemura, Epidemiological model of typhoid fever
  and its use in the planning and evaluation of antityphoid immunization and
  sanitation programmes, Bulletin of the World Health Organization 45~(1)
  (1971) 53.

\bibitem{watson2015review}
C.~H. Watson, W.~J. Edmunds, A review of typhoid fever transmission dynamic
  models and economic evaluations of vaccination, Vaccine 33 (2015) C42--C54.

\bibitem{codecco2008stochastic}
C.~T. Code{\c{c}}o, S.~Lele, M.~Pascual, M.~Bouma, A.~I. Ko, A stochastic model
  for ecological systems with strong nonlinear response to environmental
  drivers: application to two water-borne diseases, Journal of The Royal
  Society Interface 5~(19) (2008) 247--252.

\bibitem{novozhilov2006biological}
A.~S. Novozhilov, G.~P. Karev, E.~V. Koonin, Biological applications of the
  theory of birth-and-death processes, Briefings in bioinformatics 7~(1) (2006)
  70--85.

\bibitem{nee2006birth}
S.~Nee, Birth-death models in macroevolution, Annu. Rev. Ecol. Evol. Syst. 37
  (2006) 1--17.

\bibitem{thorne1991evolutionary}
J.~L. Thorne, H.~Kishino, J.~Felsenstein, An evolutionary model for maximum
  likelihood alignment of dna sequences, Journal of Molecular Evolution 33~(2)
  (1991) 114--124.

\bibitem{crawford2014estimation}
F.~W. Crawford, V.~N. Minin, M.~A. Suchard, Estimation for general birth-death
  processes, Journal of the American Statistical Association 109~(506) (2014)
  730--747.

\bibitem{moran1953estimation}
P.~Moran, The estimation of the parameters of a birth and death process,
  Journal of the Royal Statistical Society: Series B (Methodological) 15~(2)
  (1953) 241--245.

\bibitem{reynolds1973estimating}
J.~F. Reynolds, On estimating the parameters of a birth-death process,
  Australian Journal of Statistics 15~(1) (1973) 35--43.

\bibitem{keiding1975maximum}
N.~Keiding, Maximum likelihood estimation in the birth-and-death process, The
  Annals of Statistics 3~(2) (1975) 363--372.

\bibitem{perkins2009maximum}
T.~Perkins, Maximum likelihood trajectories for continuous-time markov chains,
  Advances in neural information processing systems 22 (2009).

\bibitem{immel1951problems}
E.~R. Immel, Problems of estimation and hypothesis testing in connection with
  birth-and-death stochastic processes, Ph.D. thesis, Doctoral Thesis,
  University of California, Los Angeles. R{\'e}sum{\'e}: Ann. Math~… (1951).

\bibitem{crawford2012transition}
F.~W. Crawford, M.~A. Suchard, Transition probabilities for general
  birth--death processes with applications in ecology, genetics, and evolution,
  Journal of mathematical biology 65~(3) (2012) 553--580.

\bibitem{xu2015likelihood}
J.~Xu, P.~Guttorp, M.~Kato-Maeda, V.~N. Minin, Likelihood-based inference for
  discretely observed birth--death-shift processes, with applications to
  evolution of mobile genetic elements, Biometrics 71~(4) (2015) 1009--1021.

\bibitem{davison2021parameter}
A.~C. Davison, S.~Hautphenne, A.~Kraus, Parameter estimation for discretely
  observed linear birth-and-death processes, Biometrics 77~(1) (2021) 186--196.

\bibitem{baum1966statistical}
L.~E. Baum, T.~Petrie, Statistical inference for probabilistic functions of
  finite state markov chains, The annals of mathematical statistics 37~(6)
  (1966) 1554--1563.

\bibitem{feller1971introduction}
W.~Feller, An introduction to probability theory and its applications, Wiley
  series in probability and mathematical statistics, 1971.

\bibitem{kobayashi2021stochastic1}
H.~Kobayashi, Stochastic modeling of an infectious disease part ii: Simulation
  experiments and verification of the analysis, arXiv preprint arXiv:2101.11394
  (2021).

\bibitem{ross2014introduction}
S.~M. Ross, Introduction to probability models, Academic press, 2014.

\bibitem{corless1993dj}
R.~M. Corless, G.~H. Gonnet, D.~Hare, On the lambert w function, Advances in
  Computational Mathematics 5 (1996) 329--359.

\bibitem{corless1993lambert}
R.~M. Corless, G.~H. Gonnet, D.~E. Hare, D.~J. Jeffrey, D.~Knuth, Lambert's w
  function in maple, Maple Technical Newsletter 9~(1) (1993) 12--22.

\bibitem{teodorescu2009maximum}
I.~Teodorescu, Maximum likelihood estimation for markov chains, arXiv preprint
  arXiv:0905.4131 (2009).

\bibitem{fritzPartial}
J.~Fritz, J.~LaSalle, L.~Sirovich, Partial differential equations, Applied
  Mathematical Sciences, 1977.

\bibitem{tavare2018linear}
S.~Tavar{\'e}, The linear birth--death process: an inferential retrospective,
  Advances in Applied Probability 50~(A) (2018) 253--269.

\bibitem{bailey1991elements}
N.~T. Bailey, The elements of stochastic processes with applications to the
  natural sciences, Vol.~25, John Wiley \& Sons, 1991.

\bibitem{Goerg}
G.~M. Goerg, LambertW: Probabilistic Models to Analyze and Gaussianize
  Heavy-Tailed, Skewed Data, r package version 0.6.6 (2020).

\bibitem{himmelmann2022package}
M.~L. Himmelmann, Package ‘hmm’ (2022).

\bibitem{athreya1992bootstrapping}
K.~Athreya, C.~Fuh, Bootstrapping markov chains: countable case, Journal of
  Statistical Planning and Inference 33~(3) (1992) 311--331.

\bibitem{spedicato2016markovchain}
G.~Spedicato, T.~Kang, S.~Yalamanchi, M.~Thoralf, D.~Yadav, N.~Castillo,
  V.~Jain, Markovchain: easy handling discrete time markov chains (2016).

\bibitem{coppolalbfgs}
A.~Coppola, B.~Stewart, lbfgs: Efficient l-bfgs and owl-qn optimization in r
  (2020).

\bibitem{pitzer2015mathematical}
V.~E. Pitzer, N.~A. Feasey, C.~Msefula, J.~Mallewa, N.~Kennedy, Q.~Dube,
  B.~Denis, M.~A. Gordon, R.~S. Heyderman, Mathematical modeling to assess the
  drivers of the recent emergence of typhoid fever in blantyre, malawi,
  Clinical Infectious Diseases 61~(suppl\_4) (2015) S251--S258.

\bibitem{abboubakar2021mathematical}
H.~Abboubakar, R.~Racke, Mathematical modeling, forecasting, and optimal
  control of typhoid fever transmission dynamics, Chaos, Solitons \& Fractals
  149 (2021) 111074.

\end{thebibliography}
\appendix
%
\section{Jacobian matrix of mapping $g$}
\label{app:D}
%
We compute the Jacobian matrix $Dg$ of mapping $g$ defined by Eq. (\ref{eq:defg}) in Theorem \ref{main1}. Firstly, let us compute the partial derivatives of the functions $q$, $r$ and $u$ from the proof of Theorem \ref{main1} given in Section \ref{app:th2} above. One has 
\begin{align*}
 \frac{\partial q}{\partial p_{0,0}} &= -\frac{p_{0,1} \left(\ln(p_{0,0} ) +1 \right)}{\left(p_{0,0}\ln(p_{0,0} )\right)^2} W\left( \frac{p_{0,0}^{\left(\frac{p_{0,0}}{p_{0,1}} +1 \right)} \ln(p_{0,0} )}{p_{0,1}} \right)\\ 
 & + \frac{p_{0,1}}{p_{0,0}\ln(p_{0,0})} \left(\ln(p_{0,0})+1 \right) \left( \frac{1}{p_{0,1}} + \frac{1}{p_{0,0}\ln(p_{0,0})} \right) \frac{W\left( \frac{p_{0,0}^{\left(\frac{p_{0,0}}{p_{0,1}} +1 \right)} \ln(p_{0,0} )}{p_{0,1}} \right)}{W\left( \frac{p_{0,0}^{\left(\frac{p_{0,0}}{p_{0,1}} +1 \right)} \ln(p_{0,0} )}{p_{0,1}} \right) + 1}. 
\end{align*}
Since 
\begin{align}\label{eq23}
  W\left( \frac{p_{0,0}^{\left(\frac{p_{0,0}}{p_{0,1}} +1 \right)} \ln(p_{0,0} )}{p_{0,1}} \right) = q \frac{p_{0,0}\ln(p_{0,0})}{p_{0,1}},
\end{align}
one has
\begin{align*}
 \frac{\partial q}{\partial p_{0,0}} 
 &= \frac{q \left(\ln(p_{0,0})+1 \right)}{q p_{0,0}\ln(p_{0,0})+ p_{0,1}} \left( 1 + \frac{p_{0,1}}{p_{0,0}\ln(p_{0,0})}-\frac{q p_{0,0}\ln(p_{0,0})+ p_{0,1}}{p_{0,0}\ln(p_{0,0})} \right) \\ 
 &= \frac{q \left(\ln(p_{0,0})+1 \right)}{q p_{0,0}\ln(p_{0,0})+ p_{0,1}} \left( 1- q \right).
\end{align*}
Next, on has
\begin{align*}
 \frac{\partial q}{\partial p_{0,1}} 
 &= \frac{1}{p_{0,0}\ln(p_{0,0} )} W\left( \frac{p_{0,0}^{\left(\frac{p_{0,0}}{p_{0,1}} +1 \right)}\ln(p_{0,0} )}{p_{0,1}} \right)\\
 & -\frac{p_{0,1}}{p_{0,0}\ln(p_{0,0})} \left( \frac{p_{0,0}\ln(p_{0,0} )}{(p_{0,1})^2} +\frac{1}{p_{0,1}} \right)\frac{W\left( \frac{p_{0,0}^{\left(\frac{p_{0,0}}{p_{0,1}} +1 \right)}\ln(p_{0,0} )}{p_{0,1}} \right)}{W\left( \frac{p_{0,0}^{\left(\frac{p_{0,0}}{p_{0,1}} +1 \right)}\ln(p_{0,0} )}{p_{0,1}} \right) + 1}, \end{align*}
and using Equation (\ref{eq23}) we obtain
\begin{align*}
 \frac{\partial q}{\partial p_{0,1}} 
 &= \frac{q}{p_{0,1}} 
 -\frac{q}{q p_{0,0}\ln(p_{0,0}) + p_{0,1}} \left( \frac{p_{0,0}\ln(p_{0,0} )}{p_{0,1}} + 1 \right)\\  
 &= \frac{q}{p_{0,1} \left(q +\frac{p_{0,1}} {p_{0,0}\ln(p_{0,0})} \right)} \left( q -1 \right),
\end{align*}
and finally one has
\begin{align*}
 \frac{\partial q}{\partial p_{1,0}} = 0.
\end{align*}
The expressions of partial differentials of $ r$ are given as follows
\begin{align*}
\frac{\partial r}{\partial p_{0,0}} 
&=  \frac{1}{p_{0,0}\ln(q)} - \frac{\ln(p_{0,0})}{(\ln(q))^2 q }\frac{\partial q}{\partial p_{0,0}} \\
&= \frac{r}{p_{0,0}\ln(p_{0,0})} - r \frac{ \left(\ln(p_{0,0})+1 \right) \left( 1- q \right)}{\ln(q) \left( q p_{0,0}\ln(p_{0,0})+ p_{0,1} \right)}, \\
\frac{\partial r}{\partial p_{0,1}}
&= \frac{\ln(p_{0,0})}{(\ln(q))^2 q} \frac{\partial q}{\partial p_{0,1}} \\
&= r \frac{\left( q -1 \right)}{p_{0,1}\ln(q) \left(q +\frac{p_{0,1}} {p_{0,0}\ln(p_{0,0})} \right)}, \\
\frac{\partial r}{\partial p_{1,0}} &=0.
\end{align*}
This yields the following expressions of partial differentials of $ u$: 
\begin{align*}
\frac{\partial u}{\partial p_{0,0}} &= \frac{p_{1,0}}{p_{0,0}} = \frac{1- u}{p_{0,0}}, \\
\frac{\partial u}{\partial p_{0,1}} &=0, \\ 
\frac{\partial u}{\partial p_{1,0}} &= -\frac{1}{p_{0,0}}.
\end{align*}
Defining the following functions 
\begin{align*}
  \alpha &= \frac{p \left(\ln(p_{0,0} ) +1 \right)\left(p -1 \right)}{\Delta t \left( q \ p_{0,0}\ln(p_{0,0}) + p_{0,1} \right) \left( q-u \right)} \\ \nonumber          
 \beta &= \frac{1-p}{\Delta t p_{0,1} \left( q  + \frac{p_{0,1}}{ p_{0,0}\ln(p_{0,0})}  \right) \left( p-u \right)} \\ \nonumber
 \eta  &= \frac{1 -u}{\Delta t p_{0,0} u \left( q-u \right)},
\end{align*}
one can rewrite compactly  the partial differentials of $g_1$, $g_2$ and $g_3$ (i.e. of $g$) as follows:  
\begin{align*}
\frac{\partial g_1}{\partial p_{0,0}} 
&= \lambda \left( 1-u \right) \left( \frac{\alpha \Delta t  }{1-q}+ \frac{1}{p_{0,0} \left(q-u \right)} \right) - \left( 1-q \right) \left(\frac{\alpha }{q}+ \eta \right),  \\ 
\frac{\partial g_1}{\partial p_{0,1}} 
&=  \lambda \frac{\beta \Delta t  q \left(1-u \right)}{1-q} - \beta \left( 1-q \right), \\ 
\frac{\partial g_1}{\partial p_{1,0}} 
&=  -  \frac{\lambda}{p_{0,0} \left( q-u \right) }+ \eta \dfrac{1-q}{1-u}, \\ 
\frac{\partial g_2}{\partial p_{0,0}} 
&=  \mu \left(\alpha \Delta t  + \dfrac{1-q}{p_{0,0} \left(q-u \right)} \right) - \left( 1-u \right) \left( \frac{\alpha}{q} + \eta \right), \\
\frac{\partial g_2}{\partial p_{1,0}} 
&= \mu \beta \Delta t q - \beta \left( 1-u \right),\\ 
\frac{\partial g_2}{\partial p_{0,1}} 
&=  \mu \frac{1-q}{p_{0,0} \left(u-1 \right) \left(q-u \right)} + \eta, \\
\frac{\partial g_3}{\partial p_{0,0}} 
&=  \nu \left(\frac{1}{p_{0,0} \ln\left( p_{0,0} \right)} + \frac{\alpha \Delta t   \left( q-u \right)}{q \ln\left( q \right)} +\left( 1-u \right) \left( \frac{\alpha \Delta t  }{1-q}+ \frac{1}{p_{0,0} \left(q-u \right)} \right) \right) \\ 
&\quad - r \left(1-q \right) \left(  \frac{\alpha }{q} + \eta \right), \\
\frac{\partial g_3}{\partial p_{0,1}} 
&= \nu \beta \Delta t \left( \frac{q \left(1-u \right)}{1-q} - \dfrac{q-u}{\ln\left(q \right)} \right) - r \beta \left( 1- q \right), \\ 
\frac{\partial g_3}{\partial p_{1,0}} 
&= -  \frac{\nu}{p_{0,0} \left( q-u \right) }+ r \eta \frac{1-q}{1-u}. 
\end{align*}
%
\section{Numerical results}
\label{app:num}
%
We display here the various results of our numerical investigations in the form of tables instead of graphs, for the sake of completeness. All tables are commented in Section \ref{sec:num-simu} or \ref{sec:dataARS}.
%
\begin{table}[p]
\caption{Asymptotic standard deviation of the normalized estimators (synthetic data, infected observed, true parameter values $\lambda=0.03 $, $ \mu=0.1 $ and $\nu=0.01$)}
\label{table-sigma}
\begin{center}
\begin{tabular}{@{}cccc@{}}
\toprule
time lapse $\Delta t$&$1$&$7$&$30$\\
\midrule
$\Sigma_{11}^{1/2}$  & 3.190  &0.632 &0.377 \\
$\Sigma_{22}^{1/2}$  & 1.112  & 0.552 & 0.749 \\
$\Sigma_{33}^{1/2}$  & 15.309  & 1.264 &0.443  \\
\hline
\end{tabular}
\end{center}
\end{table}
%
\begin{table}[p]
\caption{Impact of the choice of the truncation parameter $N$ on the estimation (with empirical standard errors) using HMM with analytic parameter inversion for $100$ replications (synthetic data, infected hidden, true parameter values $\lambda=0.03 $, $ \mu=0.1 $ and $\nu=0.01$)}
\label{table2}
\begin{center}
\begin{tabular}{@{}cccc@{}}
\toprule
Truncation parameter $N$&$3$&$4$&$5$ \\
\midrule
$\hat\lambda^n$  & $0.031\ (4.10^{-3})$ &$ 0.030\  (4.10^{-3})$&$0.030\ (4.10^{-3})$ \\
$\hat\mu^n$  & $0.104\ \ \ (10^{-3})$  & $0.101\ \ \ (10^{-3})$ & $0.100\ \ \ (10^{-3})$\\
$\hat\nu^n$ & $0.010\   \ \ (10^{-4})$    & $ 0.010\ \ \ (10^{-4})$  &$0.010\ \ \ (10^{-4})$ \\
\hline
\end{tabular}
\end{center}
\end{table}
%
\begin{table}[p]
\caption{The impact of the time horizon $H$ on the estimation for a truncation parameter $N=5$ (with empirical standard errors)  using HMM with analytic parameter inversion for $100$ replications (synthetic data, infected hidden, true parameter values $\lambda=0.03 $, $ \mu=0.1 $ and $\nu=0.01$)}
\label{table3}
\begin{center}
\begin{tabular}{@{}ccccc@{}}
\toprule
time horizon $H$&$1000$&$2000$&$3000$&$5000$ \\
number of observations $n$&$1001$&$2001$&$3001$& $5001$ \\
\midrule
$\hat\lambda^n$  & $0.032\ \ \  (10^{-2}) $ & $0.028\  (6.10^{-3})$& $0.031\ (5.10^{-3})$ & $0.030\ (4.10^{-3})$\\
$\hat\mu^n$  & $0.117\  (5.10^{-3}) $  &$ 0.108\  (3.10^{-3})$ & $0.103\ (2.10^{-3})$&$0.100\ \ \  (10^{-3})$\\
$\hat\nu^n$ & $0.009\  (3.10^{-4})$   & $0.010\  (2.10^{-4})$  & $0.010\ (2.10^{-4})$& $0.010\ \ \  (10^{-4})$\\
\hline
\end{tabular}
\end{center}
\end{table}
%
\begin{table}[p]
\caption{Impact of the time horizon $H$ on parameter estimations and their confidence intervals using HMM with least-squares approximation for $100$ replications (synthetic data, infected hidden, true parameter values $\lambda=0.03 $, $ \mu=0.1 $ and $\nu=0.01$)}
\label{table5.1}
\begin{center}
\begin{tabular}{@{}ccccc@{}}
\toprule
time horizon $H$&$1000$&$2000$&$3000$&$5000$ \\
number of observations $n$&$1001$&$2001$&$3001$& $5001$ \\
\midrule
$\hat\lambda^n$  & $0.034\ (8.10^{-4}) $ & $0.033\  (5.10^{-4})$& $0.032\  (5.10^{-4})$ & $0.031\  (5.10^{-4})$\\
$\hat\mu^n$  & $0.120\  (5.10^{-3}) $  &$ 0.112\  (4.10^{-3})$ & $0.111\  (3.10^{-3})$&$0.107\  (2.10^{-3})$\\
$\hat\nu^n$ & $0.013\  (9.10^{-4})$   & $0.012\  (8.10^{-4})$  & $0.011\  (5.10^{-4})$& $0.010\  (4.10^{-4})$\\
\hline
\end{tabular}
\end{center}
\end{table}
%
\begin{table}[p]
\caption{Impact of the truncation parameter $N$ on parameter estimations and their confidence intervals using HMM with least-squares approximation  for $100$ replications ($H=5000$, synthetic data, infected hidden, true parameter values $\lambda=0.03 $, $ \mu=0.1 $ and $\nu=0.01$)}
\label{table:5.2}
\begin{center}
\begin{tabular}{@{}cccc@{}}
\toprule
Truncation parameter $N$&$3$&$4$&$5$ \\
\midrule
$\hat\lambda^n$  & $0.030\  (4.10^{-4})$ &$ 0.031\   (4.10^{-4})$&$0.031\  (5.10^{-4})$ \\
$\hat\mu^n$  & $0.100\ \ \   (10^{-3})$  & $0.100\  (2.10^{-3})$ & $0.107\  (2.10^{-3})$\\
$\hat\nu^n$ & $0.010\ \ \    (10^{-4})$    & $ 0.010\   (2.10^{-4})$  &$0.010\  (4.10^{-4})$ \\
\hline
\end{tabular}
\end{center}
\end{table}
%
\begin{table}[p]
\caption{Estimates $\hat\lambda^n,\hat\mu^{n}$ and $\hat\nu^{n}$ and log likelihood of the corresponding models for the different initial parameter values tested for Mayotte typhoid fever data set}
\label{table5}
\begin{center}
\begin{tabular}{@{}ccc@{}}
\toprule
$(\lambda^{(0)},\mu^{(0)},\nu^{(0)})$&$(\hat\lambda^{n},\hat\mu^{n},\hat\nu^{n})$& log-likelihood\\
\midrule
$(0.0500, 0.1100, 0.0150)$  & $(0.0500, 0.1100, 0.0150)$ & $ -422.8953$\\
$(0.0516, 0.1174, 0.0157)$ & $(0.0515, 0.1172, 0.0159)$  & $-423.0785$ \\
$(0.0532, 0.1247, 0.0166)$& $(0.0530, 0.1242, 0.0167)$    & $-423.1790$\\
$(0.0546, 0.1321, 0.0174)$& $(0.0547, 0.1322, 0.0171)$    &$\mathbf{-422.3025}$\\
$(0.0563, 0.1395, 0.0182)$& $(0.0561, 0.1402, 0.0181)$    & $-423.3840$\\
$(0.0579, 0.1468, 0.0189)$& $(0.0578, 0.1470, 0.0188)$    & $-423.4823$\\
$(0.0595, 0.1542, 0.0197)$& $(0.0594, 0.1540, 0.0200)$    & $-423.6156$\\
$(0.0611, 0.1616, 0.0205)$& $(0.0611, 0.1615, 0.0205)$    & $ -423.7838$\\
$(0.0626, 0.1689, 0.0213)$& $(0.0628, 0.1690, 0.0214)$    & $-423.9202$\\
$(0.0642, 0.1763, 0.0221)$& $(0.0643, 0.1762, 0.0220)$    & $-424.0504$\\
$(0.0658, 0.1837, 0.0229)$& $(0.0659, 0.1836, 0.0227)$    & $-423.7441$\\
$(0.0674, 0.1911, 0.0237)$& $(0.0673, 0.1910, 0.0235)$    & $-424.0065$\\
$(0.0689, 0.1984, 0.0245)$& $(0.0688, 0.1983, 0.0242)$    & $-424.3249$\\
$(0.0705, 0.2058, 0.0253)$& $(0.0702, 0.2057, 0.0252)$    & $-424.5890$\\
$(0.0721, 0.2132, 0.0261)$& $(0.0720, 0.2132, 0.0262)$    & $-425.1636$\\
$(0.0737, 0.2205, 0.0268)$& $(0.0742, 0.2204, 0.0270)$    & $-425.5052$\\
$(0.0753, 0.2279, 0.0276)$& $(0.0751, 0.2280, 0.0275)$    & $-426.0368$\\
$(0.0768, 0.2353, 0.0284)$& $(0.0769, 0.2352, 0.0281)$    & $ -426.2674$\\
$(0.0784, 0.2426, 0.0292)$& $(0.0785, 0.2424, 0.0293)$    & $ -426.5001$\\
$(0.0800, 0.2500, 0.0300)$& $(0.0800, 0.2500, 0.0300)$    & $-426.7330$\\
\hline
\end{tabular}
\end{center}
\end{table}

\section*{Declarations}
The authors declare that they have no competing interests. 
All co-authors have seen and agree with the contents of the manuscript. We certify that the submission is original work and is not under review at any other publication. 

\subsection*{Availability of data and materials}
The datasets used and/or analyzed and the codes developped during the current study available at \url{https://plmlab.math.cnrs.fr/bouzalma/hlbdi.git}

\subsection*{Acknowledgments}
We thank all of the people involved in this work especially Julien Balicchi and the DéSUS department of the ARS Mayotte for the data collection and availability. This project was partially supported by the LabEx NUMEV (ANR 2011-LABX-076) within the I-SITE MUSE.

\end{document}